\documentclass[12pt]{amsart}
\usepackage[utf8]{inputenc}
\usepackage{amsfonts, amstext, amscd, amsmath, mathrsfs, amssymb, amsthm, amsxtra, bbding, epsfig, eucal, eufrak, graphicx, latexsym, url, color}
\usepackage{bbold}
\usepackage[all]{xy}
\usepackage{tikz}

\usepackage[normalem]{ulem}
\usepackage{soul}
\usepackage{color}

\setstcolor{red}

\makeatletter
\@namedef{subjclassname@2010}{%
  \textup{2010} Mathematics Subject Classification}
\makeatother

\def \C {{\mathbb C}}

\def \N {{\mathbb N}}

\def \R {{\mathbb R}}

\def \d {\,{\rm d}}
\def\re{{\Re e\,}}
\def\im{{\Im m\,}}

\def \sset {{\smallsetminus }}

\def\leq{\leqslant}
\def\geq{\geqslant}
\def\le{\leqslant}
\def\ge{\geqslant}

\theoremstyle{plain}
\newtheorem{theorem}{Theorem}[section]
\newtheorem{proposition}{Proposition}[section]
\newtheorem{lemma}{Lemma}[section]
\newtheorem{corollary}{Corollary}[section]

\theoremstyle{remark}

\numberwithin{equation}{section}

\addtocounter{footnote}{1}

\begin{document}

\title[Mean values for a class of arithmetic functions in short intervals]
{Mean values for a class of arithmetic functions
\\
in short intervals}
\author{Jie Wu and Qiang Wu}

\address{%
CNRS, UMR 8050\\
Laboratoire d'Analyse et de Math\'ematiques Appliqu\'ees\\
Universit\'e Paris-Est Cr\'eteil\\
61 Avenue du G\'en\'eral de Gaulle\\
94010 Cr\'eteil cedex\\
France
}
\email{jie.wu@math.cnrs.fr}

\address{%
Qiang Wu
\\
Department of Mathematics
\\
Southwest University of China
\\
2 Tiansheng Road
\\
Beibei
\\
400715 Chongqing
\\
China}
\email{qiangwu@swu.edu.cn}

\date{\today}

\begin{abstract}
In this paper, we shall establish a rather general asymptotic formula in short intervals for a classe of arithmetic functions and announce two applications about the distribution of divisors of square-full numbers 
and integers representable as sums of two squares.
\end{abstract}
\maketitle


\section{Introduction}

This is the third paper of our series  on the Selberg-Delange method for \textit{short intervals}. 
Roughly speaking this method applies to evaluate mean values of arithmetic functions whose associated Dirichlet series are close to \textit{complex} powers of the Riemann $\zeta$-function.
In the first part, by using a suitable contour 
(located to the left of the Korobov-Vinogradov zero-free region of $\zeta(s)$) 
instead of Hankel's contour used in the original version of the Selberg-Delange method, 
Cui and Wu \cite{CuiWu2014} extended this method to handle mean values of arithmetic functions over a short interval, when their corresponding Dirichlet series is close to a positive power of $\zeta$-function.
In the second one, Cui, L\"u and Wu \cite{CuiLvWu2018} treated the complex power case 
with the help of the well-known Hooley-Huxley-Motohashi contour.
Some similar results have were appeared in Ramachandra's paper \cite{Ramachandra1976}.
In this paper, we shall consider a more general case and give some arithmetic applications.

\subsection{Assumptions}\

\vskip 0,5mm

Let us fix some notation:
\begin{itemize}
\item[--]
$\zeta(s)$ is the Riemann $\zeta$-function,
\item[--]
$L(s, \chi)$ is the Dirichlet $L$-function of $\chi$,
\item[--]
$\varepsilon$ is an arbitrarily small positive constant,
\item[--]
$r\in \N$, $\alpha>0$, $\delta\ge 0$, $A\ge 0$, $M>0$ (constants),
\item[--]
$\boldsymbol{z} := (z_1, \dots, z_r)\in \C^r$ and $\boldsymbol{w} := (w_1, \dots, w_r)\in \C^r$,
\item[--]
$\boldsymbol{\kappa} := (\kappa_1, \dots, \kappa_r)\in (\R^{+*})^r$ 
with $1\le \kappa_1<\dots <\kappa_r\le 2\kappa_1$,
\item[--]
$\boldsymbol{\chi} := (\chi_1, \dots, \chi_r)$ with $\chi_i$ non principal Dirichlet characters,
\item[--]
$\boldsymbol{B} := (B_1, \dots, B_r)\in (\R^{+*})^r$ and $\boldsymbol{C} := (C_1, \dots, C_r)\in (\R^{+*})^r$,
\item[--]
The notation $|\boldsymbol{z}|\le \boldsymbol{B}$ means that $|z_i|\le B_i$ for $1\le i\le r$.
\end{itemize}
Let $f : \N\to \C$ be an arithmetic function and its corresponding Dirichlet series is given by
\begin{equation}\label{def:Fs}
{\mathcal F}(s)
:= \sum_{n=1}^{\infty} f(n) n^{-s}.
\end{equation}
We say that this Dirichlet series ${\mathcal F}(s)$ is of type
${\mathcal P}(\boldsymbol{\kappa}, \boldsymbol{z}, \boldsymbol{w}, 
\boldsymbol{B}, \boldsymbol{C}, \alpha, \delta, A, M)$
if the following conditions are verified:
\par
\goodbreak
(a)
For any $\varepsilon>0$ we have
\begin{equation}\label{UBf(n)}
|f(n)|\ll_{\varepsilon} M n^{\varepsilon}
\qquad
(n\ge 1),
\end{equation}
where the implied constant depends only on $\varepsilon$;
\par
(b)
We have
$$
\sum_{n=1}^{\infty} |f(n)|n^{-\sigma}
\le M (\sigma-1/\kappa_1)^{-\alpha}
\qquad
(\sigma>1/\kappa_1);
$$
\par
(c)
The Dirichlet series $\mathcal{F}(s)$ has the expression
\begin{equation}\label{Expression:Fs}
{\mathcal F}(s) 
= \boldsymbol{\zeta}(\boldsymbol{\kappa}s)^{\boldsymbol{z}}
\boldsymbol{L}(\boldsymbol{\kappa}s; \boldsymbol{\chi})^{\boldsymbol{w}}
\mathcal{G}(s),
\end{equation}
where
\begin{align}
\boldsymbol{\zeta}(\boldsymbol{\kappa}s)^{\boldsymbol{z}} 
& := \prod_{1\le i\le r} \zeta(\kappa_i s)^{z_i},
\label{def:zetakappasz}
\\
\boldsymbol{L}(\boldsymbol{\kappa}s; \boldsymbol{\chi})^{\boldsymbol{w}} 
& := \prod_{1\le i\le r} L(\kappa_i s, \chi_i)^{w_i}
\label{def:Lkappaschiw}
\end{align}
and the Dirichlet series $\mathcal{G}(s)$
is a holomorphic function in (some open set containing) $\sigma\ge (2\kappa_1)^{-1}$ and,
in this region, $\mathcal{G}(s)$ satisfies the bound
\begin{equation}\label{UBGszw}
|\mathcal{G}(s)|
\le M(|\tau|+1)^{\max\{\delta (1-\kappa_1\sigma), 0\}}\log^A(|\tau|+3)
\end{equation}
uniformly for $|\boldsymbol{z}|\le \boldsymbol{B}$ and $|\boldsymbol{w}|\le \boldsymbol{C}$,
where and in the sequel we implicitly define the real numbers $\sigma$ and $\tau$ by the relation $s = \sigma + \text{i}\tau$
and choose the principal value of the complex logarithm.

As usual, we denote by $N(\sigma, T)$ and $N_{\chi}(\sigma, T)$ the number of zeros of $\zeta(s)$ and $L(s, \chi)$
in the region $\re s\ge \sigma$ and $|\im s|\le T$, respectively.
It is well known that there are two constants $\psi$ and $\eta$ such that
\begin{equation}\label{ZeroDensity}
N(\sigma, T), \; N_{\chi}(\sigma, T)
\ll T^{\psi(1-\sigma)} (\log T)^{\eta}
\end{equation}
for $\tfrac{1}{2}\le \sigma\le 1$ and $T\ge 2$.
Huxley \cite{Huxley1972} showed that
\begin{equation}\label{Huxley12/5}
\psi = \tfrac{12}{5}
\qquad\text{and}\qquad
\eta=9
\end{equation}
are admissible for $\zeta(s)$.
It is not difficult to extend it for $L(s, \chi)$. 
The  zero density hypothesis is stated as
\begin{equation}\label{ZeroDensityHuxley}
\psi = 2.
\end{equation}

\subsection{Set-up and main results}\

Our main aim of this paper is to establish, under the previous assumptions, an asymptotic formula for the summatory function 
\begin{equation}\label{suminshortinterval}
\sum_{x<n\le x+x^{1-1/\kappa_1}y} f(n)
\end{equation}
where $y:=x^{\vartheta}$ with $\vartheta\in (0, 1]$ as small as possible.
In order to state our result, it is necessary to introduce some notation.
Obviously the function $Z(\kappa s; z) :=  \{(\kappa s-1)\zeta(\kappa s)\}^{z}$
is holomorphic in the disc $|s-1/\kappa |<1/\kappa$, and admits, in the same disc,  the Taylor series expansion
$$
Z(\kappa s; z)
= \sum_{j=0}^{\infty} \frac{\gamma_j(z, \kappa )}{j !} (s-1/\kappa )^j,
$$
where the $\gamma_j(z, \kappa )$'s are entire functions of $z$ satisfying the estimate
\begin{equation}\label{UBgammajz}
\frac{\gamma_j(z, \kappa)}{j !}
\ll_{B, \kappa, \varepsilon} (1+\varepsilon)^j
\qquad
(j\ge 0, \, |z|\le B)
\end{equation}
for all $B>0$, $\kappa$ and $\varepsilon>0$. 
Write $\boldsymbol{\kappa}_*=(\kappa_2, \dots, \kappa_r)$ and $\boldsymbol{z}_*=(z_2, \dots, z_r)$.
Under our hypothesis, the function 
$$
\mathcal{G}(s) 
Z(\kappa_1s; z_1) \boldsymbol{\zeta}(\boldsymbol{\kappa}_*s)^{\boldsymbol{z}_*}
\boldsymbol{L}(\boldsymbol{\kappa}s; \boldsymbol{\chi})^{\boldsymbol{w}}
$$
is holomorphic in the disc $|s-1/\kappa_1|<1/\kappa_1-1/\kappa_2$ and
\begin{equation}\label{UBGzetaZ}
|Z(\kappa_1s; z_1) \boldsymbol{\zeta}(\boldsymbol{\kappa}_*s)^{\boldsymbol{z}_*}
\boldsymbol{L}(\boldsymbol{\kappa}s; \boldsymbol{\chi})^{\boldsymbol{w}} \mathcal{G}(s)|
\ll_{A, \boldsymbol{B}, \boldsymbol{C}, \delta, \varepsilon} M
\end{equation}
for $|s-1/\kappa_1|<1/\kappa_1-1/\kappa_2-\varepsilon$, 
$|\boldsymbol{z}|\le \boldsymbol{B}$ and $|\boldsymbol{w}|\le \boldsymbol{C}$.
Thus we can write
\begin{equation}\label{TaylorSeriesGzetaZ} 
Z(\kappa_1s; z_1) \boldsymbol{\zeta}(\boldsymbol{\kappa}_*s)^{\boldsymbol{z}_*}
\boldsymbol{L}(\boldsymbol{\kappa}s; \boldsymbol{\chi})^{\boldsymbol{w}}
\mathcal{G}(s)
= \sum_{\ell=0}^{\infty} g_{\ell}(\boldsymbol{\kappa}, \boldsymbol{z}, \boldsymbol{w}, \boldsymbol{\chi}) (s-1/\kappa_1)^{\ell}
\end{equation}
for $|s-1/\kappa_1|\le \tfrac{1}{2}(1/\kappa_1-1/\kappa_2)$,
where
\begin{equation}\label{defgkzw}
g_{\ell}(\boldsymbol{\kappa}, \boldsymbol{z}, \boldsymbol{w}, \boldsymbol{\chi})
:= \frac{1}{\ell !} \sum_{j=0}^{\ell} \binom{\ell}{j} \gamma_j(z_1, \kappa_1)
\frac{\partial^{\ell-j}
(\boldsymbol{\zeta}(\boldsymbol{\kappa}_*s)^{\boldsymbol{z}_*}
\boldsymbol{L}(\boldsymbol{\kappa}s; \boldsymbol{\chi})^{\boldsymbol{w}}\mathcal{G}(s))}{\partial s^{\ell-j}} \bigg|_{s=1/\kappa_1}.
\end{equation}

The main result of this paper is as follows.

\begin{theorem}\label{thm1.1}
Let $r\in \N$, 
$\boldsymbol{\kappa}\in (\R^{+*})^r$,
$\boldsymbol{z}\in \C^r$, 
$\boldsymbol{w}\in \C^r$, 
$\boldsymbol{B}\in (\R^{+*})^r$,
$\boldsymbol{C}\in (\R^{+*})^r$,
$\boldsymbol{\chi}$,
$\alpha>0$, $\delta\ge 0$, $A\ge 0$, $M>0$
be given as before.
Suppose that the Dirichlet series ${\mathcal F}(s)$ defined as in \eqref{def:Fs} is of type 
${\mathcal P}(\boldsymbol{\kappa}, \boldsymbol{z}, \boldsymbol{w}, \boldsymbol{B}, \boldsymbol{C}, 
\boldsymbol{\chi}, \alpha, \delta, A, M)$.
Then for any $\varepsilon>0$,  we have
\begin{equation}\label{AsymtoticFormula}
\sum_{x<n\le x+x^{1-1/\kappa_1}y} f(n)
= y'(\log x)^{z-1}
\bigg\{\sum_{\ell=0}^{N} \frac{\lambda_{\ell}(\boldsymbol{\kappa}, \boldsymbol{z}, \boldsymbol{w}, \boldsymbol{\chi})}{(\log x)^{\ell}} 
+ O\big(MR_N(x)\big)\bigg\}
\end{equation}
uniformly for
$$
x\ge 3,
\quad\,
x^{(1-1/(\psi+\delta))/\kappa_1+\varepsilon}\le y\le x^{1/\kappa_1},
\quad\,
N\ge 0,
\quad\,
|\boldsymbol{z}|\le \boldsymbol{B},
\quad\,
|\boldsymbol{w}|\le \boldsymbol{C},
$$
where 
\begin{align*}
& y' := \kappa_1 ((x+x^{1-1/\kappa_1}y)^{1/\kappa_1} - x^{1/\kappa_1}),
\\\noalign{\vskip 1mm}
& \lambda_{\ell}(\boldsymbol{\kappa}, \boldsymbol{z}, \boldsymbol{w}, \boldsymbol{\chi}) 
:= \kappa_1^{-z_1}
g_{\ell}(\boldsymbol{\kappa}, \boldsymbol{z}, \boldsymbol{w}, \boldsymbol{\chi})/\Gamma(z_1-\ell),
\\\noalign{\vskip 0mm}
& R_N(x)
:= \Big(\frac{c_1N+1}{\log x}\Big)^{N+1} 
+ \frac{(c_1N+1)^{N+1}}{{\rm e}^{c_2(\log x/\log_2x)^{1/3}}}
+ \frac{y}{x^{1/\kappa_1}\log x}
\end{align*}
for some constants $c_1>0$ and $c_2>0$ depending only on $\boldsymbol{B}, \boldsymbol{C}, \delta$ 
and $\varepsilon$.
The implied constant in the $O$-term depends only on 
$\boldsymbol{\kappa}, \boldsymbol{B}, \boldsymbol{C}, \boldsymbol{\chi}, A, \alpha, \delta$
and $\varepsilon$.
In particular, $\psi=\frac{12}{5}$ is admissible.
\end{theorem}

The admissible length of short intervals in Theorem \ref{thm1.1}
depends only on the zero density constant $\psi$ of $\zeta(s)$ and $\delta$ in \eqref{UBGszw}
(for which  we take $\delta=0$ in most applications).
Its independence from the power $z$ of $\zeta(s)$ in the representation of ${\mathcal F}(s)$ seems  interesting.
Theorem \ref{thm1.1} generalizes and improves \cite[Theorem 1]{CuiWu2014} to the case of complex powers and  intervals of shorter length.

\vskip 1mm

Taking $N=0$ in Theorem \ref{thm1.1}, we obtain readily the following  corollary.

\begin{corollary}\label{cor1.1}
Under the conditions of Theorem \ref{thm1.1}, for any $\varepsilon>0$,  we have
\begin{equation}\label{cor1.2.EqA}
\sum_{x<n\le x+x^{1-1/\kappa_1}y} f(n)
= y'(\log x)^{z-1} 
\bigg\{\lambda_0(\boldsymbol{\kappa}, \boldsymbol{z}, \boldsymbol{w}, \boldsymbol{\chi}) 
+ O\bigg(\frac{M}{\log x}\bigg)\bigg\}
\end{equation}
uniformly for
$$
x\ge 3,
\quad\,
x^{(1-1/(\psi+\delta))/\kappa_1+\varepsilon}\le y\le x^{1/\kappa_1},
\quad\,
|\boldsymbol{z}|\le \boldsymbol{B},
\quad\,
|\boldsymbol{w}|\le \boldsymbol{C},
$$
where
$$
\lambda_0(\boldsymbol{\kappa}, \boldsymbol{z}, \boldsymbol{w}, \boldsymbol{\chi}) 
:= \frac{\mathcal{G}(1/\kappa_1)}{\kappa_1^{z_1}\Gamma(z_1)}
\prod_{2\le i\le r} \zeta(\kappa_i/\kappa_1)^{z_i}
\prod_{1\le i\le r} L(\kappa_i/\kappa_1, \chi_i)^{w_i}
$$
and the implied constant in the $O$-term depends only on $A, \boldsymbol{B}, \boldsymbol{C}, \alpha, \delta$ 
and $\varepsilon$.
Note that $\psi=\frac{12}{5}$ is admissible.
\end{corollary}

Takning $r=2$, $\boldsymbol{\kappa}=(1, 2)$, $\boldsymbol{z}=(z, w)$, $\boldsymbol{w}=(0, 0)$
in Theorem \ref{thm1.1} and Corollary \ref{cor1.1}, 
we can obtain Theorem 1.1 and Corollary 1.2 of Cui, L\"u \& Wu \cite{CuiLvWu2018}.

\subsection{Two arithmetical applications}\

\vskip 1mm

In order to study the distribution of divisors of integers, 
Deshouillers, Dress \& G. Tenenbaum \cite{DDT1979} introduced the random variable $D_n$, 
which takes the value $(\log d)/\log n$,
as $d$ runs through the set of the $\tau(n)$ divisors of $n$,
with the uniform probability $1/\tau(n)$, and consider its distribution function 
$$
F_n(t)
= \text{Prob}(D_n\le t)
= \frac{1}{\tau(n)} \sum_{d\mid n, \, d\le n^t} 1
\qquad
(0\le t\le 1).
$$
It is clear that the sequence $\{F_n\}_{n\ge 1}$ does not converge pointwisely on $[0, 1]$.
However Deshouillers, Dress \& Tenenbaum (\cite{DDT1979} or \cite[Theorem II.6.7]{Tenenbaum1995})
proved that its Ces\`aro mean converges uniformly to the arcsin law. 
More precisely, they showed that the asymptotic formula
\begin{equation}\label{DDT}
\frac{1}{x}\sum_{n\le x} F_n(t)
= \frac{2}{\pi} \arcsin\sqrt{t} + O\bigg(\frac{1}{\sqrt{\log x}}\bigg)
\end{equation}
holds uniformly for $x\ge 2$ and $0\le t\le 1$.
Here we shall announce two applications of Corollary \ref{cor1.1}:
the distribution of divisors of square-full numbers and of integers representable as sums of two squares.
The proof will be given in another paper \cite{WuJWuQ2018}.

\vskip 1mm

A. \textit{Beta law on divisors of square-full numbers in short intervals}\

\vskip 0,5mm

An integer $n$ is called \textit{square-full} if $p\mid n\,\Rightarrow\, p^2\mid n$.
Denote by $\mathbb{1}_{\rm sf}(n)$ the characteristic function of such integers and define
\begin{equation}\label{def:Uxy}
U(x, y) := \sum_{x<n\le x+x^{1/2}y} \mathbb{1}_{\rm sf}(n).
\end{equation}

Our first result is as follows.

\begin{theorem}\label{thm1.2}
For any $\varepsilon\in (0, \frac{5}{48})$, we have
\begin{equation}\label{eq:thm1.2}
\frac{1}{U(x, y)}\sum_{x<n\le x+x^{1/2}y} \mathbb{1}_{\rm sf}(n) F_n(t)
= B(\tfrac{2}{3}, \tfrac{1}{3})^{-1} \int_0^t w^{-\frac{1}{3}}(1-w)^{-\frac{2}{3}} \d w 
+ O\bigg(\frac{1}{\sqrt[3]{\log x}}\bigg)
\end{equation}
uniformly for $x\ge 3$ and $x^{19/48+\varepsilon}\le y\le x^{1/2}$,
where $B(u, v)$ is the beta function defined by
\begin{equation}\label{def:BetaFunction}
B(u, v) := \int_0^1 w^{u-1} (1-w)^{v-1} \d w
\end{equation}
for $u>0$ and $v>0$.
\end{theorem}

B. \textit{Beta law on divisors of integers representable as sums of two squares}

\vskip 0,5mm

Define 
\begin{equation}\label{def:sum2square}
\mathbb{1}_{\square+\square}(n)
:= \begin{cases}
1 & \text{if $n=\square+\square$}
\\
0 & \text{otherwise}
\end{cases}
\qquad\text{and}\qquad
V(x) := \sum_{n\le x} \mathbb{1}_{\square+\square}(n).
\end{equation}
They proved that the mean of $F_n(t)$ over integers representable as sum of two squares converges to the beta law
\cite[Theorem 1]{FengWu2018}:
For $x\to \infty$, we have
\begin{equation}\label{FW}
\frac{1}{V(x)}\sum_{n\le x} \mathbb{1}_{\square+\square}(n) F_n(t)
= B(\tfrac{1}{4}, \tfrac{1}{4})^{-1} \int_0^t w^{-\frac{3}{4}}(1-w)^{-\frac{3}{4}} \d w 
+ O\bigg(\frac{1}{\sqrt[4]{\log x}}\bigg).
\end{equation}
Here we shall generalise \eqref{FW} to the short interval case.
Put
\begin{equation}\label{asymp:U(xy)}
V(x, y) := \sum_{x<n\le x+y} \mathbb{1}_{\square+\square}(n).
\end{equation}

Our result is as follows.

\begin{theorem}\label{thm1.3}
For any $\varepsilon\in (0, \frac{5}{12})$, we have
\begin{equation}\label{eq:thm1.3}
\frac{1}{V(x, y)} \sum_{x<n\le x+y} \mathbb{1}_{\square+\square}(n) F_n(t)
= B(\tfrac{1}{4}, \tfrac{1}{4})^{-1} \int_0^t w^{-\frac{3}{4}}(1-w)^{-\frac{3}{4}} \d w 
+ O\bigg(\frac{1}{\sqrt[4]{\log x}}\bigg)
\end{equation}
uniformly for $x\ge 3$ and $x^{19/24+\varepsilon}\le y\le x$.
\end{theorem}

\vskip 8mm

\section{The Hooley-Huxley-Motohashi contour}\label{section2}

\vskip 1mm

The Hooley-Huxley-Motohashi contour appeared in Ramachandra \cite{Ramachandra1976} and Motohashi \cite{Motohashi1976}, independently.
In \cite{Hooley1974}, Hooley stated, without proof, his joint result with Huxley:
$$
\sum_{x<n\le x+x^{\theta}} \mathbb{1}_{\square+\square}(n)
\sim 2^{-1/2} \prod_{p\equiv 1 ({\rm mod}\,4)} (1-p^{-2})^{-1/2} \frac{x^{\theta}}{\sqrt{\log x}}
$$
for $x\to\infty$, provided $\theta>\frac{7}{12}$.
According to Ramachandra \cite[page 314]{Ramachandra1976},
Hooley and Huxley have educated him on their method and allowed him make some comments about their method.
A key point of this method is to use an ingeneous contour, which is called the Huxley-Hooley contour by Ramachandra.
In \cite{Motohashi1976}, Motohashi constructed a essential same contour to prove
$$
\sum_{x<n\le x+x^{\theta}} \mu(n)
= o(x^{\theta})
\qquad
(x\to\infty)
$$
provided $\theta>\frac{7}{12}$.
Here we shall follow Motohashi's argument and construct a Hooley-Huxley-Motohashi contour for our purpose.

\subsection{Notation}\

\vskip 1mm

Let $(\boldsymbol{\kappa}, \boldsymbol{\chi})$ be as before.
First we write
\begin{align}
\boldsymbol{\zeta}(\boldsymbol{\kappa}s)\boldsymbol{L}(\boldsymbol{\kappa}s, \boldsymbol{\chi})
& := \prod_{1\le i\le r} \zeta(\kappa_i s) L(\kappa_i s, \chi_i)
=: \sum_{n\ge 1} \tau_{\boldsymbol{\kappa}}(n; \boldsymbol{\chi}) n^{-s},
\label{def:tauan}
\\
(\boldsymbol{\zeta}(\boldsymbol{\kappa}s) \boldsymbol{L}(\boldsymbol{\kappa}s, \boldsymbol{\chi}))^{-1}
& := \prod_{1\le i\le r} (\zeta(\kappa_i s) L(\kappa_i s, \chi_i))^{-1}
=: \sum_{n\ge 1} \tau_{\boldsymbol{\kappa}}^{\langle-1\rangle}(n; \boldsymbol{\chi}) n^{-s}
\label{def:taua*n}
\end{align}
for $\sigma>1/\kappa_1$, respectively.
Here and in the sequel, we define implicitly the real numbers $\sigma$ and $\tau$ 
by the relation $s=\sigma+\text{i}\tau$.
A simple computation shows that
\begin{align}
\tau_{\boldsymbol{\kappa}}(n; \boldsymbol{\chi}) 
& = \sum_{m_1^{\kappa_1}\cdots m_r^{\kappa_r} n_1^{\kappa_1}\cdots n_r^{\kappa_r}= n} 
\prod_{1\le i\le r} \chi_i(n_i),
\label{tau23}
\\
\tau_{\boldsymbol{\kappa}}^{\langle-1\rangle}(n; \boldsymbol{\chi}) 
& = \sum_{m_1^{\kappa_1}\cdots m_r^{\kappa_r} n_1^{\kappa_1}\cdots n_r^{\kappa_r}= n} 
\prod_{1\le i\le r} \mu(m_i) \mu(n_i) \chi_i(n_i).
\label{tau23(-1)}
\end{align}
From these, we deduce that
\begin{equation}\label{(tau23)*(tau23*)}
|\tau_{\boldsymbol{\kappa}}(n; \boldsymbol{\chi})|\le \tau_{\boldsymbol{\kappa}, \boldsymbol{\kappa}}(n),
\quad
|\tau_{\boldsymbol{\kappa}}^{\langle-1\rangle}(n; \boldsymbol{\chi})|\le \tau_{\boldsymbol{\kappa}, \boldsymbol{\kappa}}(n),
\quad
\tau_{\boldsymbol{\kappa}}*\tau_{\boldsymbol{\kappa}}^{\langle-1\rangle} = \mathbb{1}_{\{1\}}
\end{equation}
for all $n\ge 1$, 
where 
\begin{equation}\label{def:taunkappakappa}
\tau_{\boldsymbol{\kappa}, \boldsymbol{\kappa}}(n) 
:= \sum_{m_1^{\kappa_1}\cdots m_r^{\kappa_r} n_1^{\kappa_1}\cdots n_r^{\kappa_r}= n} 1
\end{equation}
and $\mathbb{1}_{\{1\}}$ is the unit with respect to the Dirichlet convolution.

\subsection{Upper bounds for $\zeta(s)^{\pm 1}$ and $L(s, \chi)^{\pm 1}$ in their zero-free regions}\

\vskip 1mm

It is well known that there is an absolute positive constant $c$ such that $\zeta(s)\not=0$ for 
\begin{equation}\label{free-zero:zeta(s)}
\sigma\ge 1-c(\log|\tau|)^{-2/3}(\log_2|\tau|)^{-1/3},
\qquad
|\tau|\ge 3
\end{equation}
(the zero-free region, due to Korobov and Vinogradov)
and in this region we have
\begin{equation}\label{LBUB:zeta(s)}
\zeta(s)^{\pm 1}\ll (\log|\tau|)^{2/3}(\log_2|\tau|)^{1/3}
\end{equation}
(see \cite[page 135]{Titchmarsh1986} or \cite[page 162]{Tenenbaum1995}).
For the Dirichlet $L$-functions, Richert \cite{Richert1967} has established similar results.

\begin{lemma}\label{lem2.1}
Let $\chi$ be a non-principal Dirichlet character module $q$ and 
let $L(s, \chi)$ be the corresponding Dirichlet $L$-function.
Then we have
\begin{equation}\label{Eqlem2.1A}
|L(s, \chi)|\ll \begin{cases}
|\tau|^{100(1-\sigma)^{3/2}}(\log \tau)^{4/3}
& \text{if $\tfrac{1}{2}\le \sigma\le 1$ and $|\tau|\ge 3$},
\\\noalign{\vskip 1mm}
\log |\tau|
& \text{if $\sigma\ge 1$ and $|\tau|\ge 3$}.
\end{cases}
\end{equation}
Further there is a positive constant $c_{\chi}$ depending on $\chi$ such that 
\begin{equation}\label{Eqlem2.1B}
|L(s, \chi)|^{-1}\ll (\log |\tau|)^{2/3}(\log_2|\tau|)^{1/3}
\end{equation}
for $\sigma\ge 1-500c_{\chi} (\log|\tau|)^{-2/3}(\log_2|\tau|)^{-1/3}$ and $|\tau|\ge 3$.
Here the implied constants depend on $\chi$ only.
\end{lemma}

\begin{proof}
Let $s=\sigma+\text{i}\tau$. Without loss of generality, we can suppose that $\tau\ge 2$.
For $\sigma := \Re e\,s>1$ and $0<w\le 1$, the Hurwitz $\zeta$-function is defined by
$$
\zeta(s,w):=\sum_{n=0}^\infty (n+w)^{-s}.
$$
This function can be extended to a meromorphic function over $\C\sset \{1\}$.
According to \cite[Satz]{Richert1967}, there is a absolute constant $c>0$ such that we have
\begin{equation}\label{Richert}
\big|\zeta(s,w)-w^{-s}\big|
\le c |\tau|^{100(1-\sigma)^{3/2}}\bigl(\log |\tau|\bigr)^{2/3}
\end{equation}
uniformly for 
$0<w\le 1$,
$\frac{1}{2}\le \sigma\le 1$ and
$|\tau|\ge 3$.

Since $\chi(n)$ is of period $q$, we can write, for $\sigma>1$,
\begin{align*}
L(s,\chi)
& = \sum_{a=1}^q \sum_{n=0}^{\infty} \chi(a+nq) (nq+a)^{-s}
\\
& =q^{-s}\sum_{a=1}^q \chi(a) \big\{\zeta(s,a/q)-(a/q)^{-s}\big\}
+\sum_{a=1}^q \chi(a)a^{-s}.
\end{align*}
This relation also holds for all $s\in \C\sset\{1\}$ by analytic continuation. 
Inserting \eqref{Richert}, we immediately get the first inequality in \eqref{Eqlem2.1A}.
The second one is classical.

In view of \eqref{Eqlem2.1A}, we can prove \eqref{Eqlem2.1B} exactly as \cite[Theorem 3.11]{Titchmarsh1986} with the choice of 
$$
\theta(\tau) = \Big(\frac{\log_2\tau}{100\log\tau}\Big)^{2/3}
\qquad\text{and}\qquad
\phi(\tau) = \log_2\tau.
$$
This completes the proof of Lemma \ref{lem2.1}.

\end{proof}

\subsection{Definition of the Hooley-Huxley-Motohashi contour ${\mathscr L}_T$}\

\vskip 1mm

Let $(\varepsilon, \boldsymbol{\kappa}, \boldsymbol{\chi})$ be as before and
let $T_0 = T_0(\varepsilon, \boldsymbol{\kappa}, \boldsymbol{\chi}), 
c_0=c_0(\boldsymbol{\kappa}, \boldsymbol{\chi})$ be two large constants
and let $C_0=C_0(\boldsymbol{\kappa}, \boldsymbol{\chi})$ be a suitable positive constant.
For $T\ge T_0$, put
\begin{equation}\label{defdelta}
\delta_T := C_0 (\log T)^{-2/3}(\log_2T)^{-1/3}.
\end{equation}
According to \eqref{LBUB:zeta(s)}, \eqref{Eqlem2.1A} and \eqref{Eqlem2.1B}
where  such that
\begin{equation}\label{zero-free:zetaDirichlet}
(\log |\tau|)^{-c_0r} 
\ll |\boldsymbol{\zeta}(\boldsymbol{\kappa}s)\boldsymbol{L}(\boldsymbol{\kappa}s, \boldsymbol{\chi})|
\ll (\log |\tau|)^{c_0r}
\end{equation}
for $\sigma\ge (1-100\delta_T)/\kappa_1$ and $3\le |\tau|\le 100T$,
where the implied constants depend on $(\boldsymbol{\kappa}, \boldsymbol{\chi})$.

\vskip 1mm

For $T\ge T_0$, write
\begin{equation}\label{defJK}
J_T := [(\tfrac{1}{2}-\delta_T)\log T]
\qquad\text{and}\qquad
K_T := [T(\log T)^{-1}].
\end{equation}
For each pair of integers $(j, k)$ with $0\le j\le J_T$ and $0\le k\le K_T$, we define
\begin{equation}\label{defsigmaj}
\sigma_j := (\tfrac{1}{2}+j(\log T)^{-1})/\kappa_1,
\qquad
\tau_k := (1+k\log T)/\kappa_1
\end{equation}
and
\begin{equation}\label{defDeltajk}
\Delta_{j, k} := \{s=\sigma+\text{i}\tau : \sigma_j\le \sigma<\sigma_{j+1} \;\, \text{and}\;\, \tau_{k}\le \tau<\tau_{k+1}\}.
\end{equation}

Define
\begin{align}
M_{x}(s) 
= M_{x}(s; \boldsymbol{\kappa}, \boldsymbol{\chi})
& := \sum_{n\le x} \frac{\tau_{\boldsymbol{\kappa}}^{\langle-1\rangle}(n; \boldsymbol{\chi})}{n^s},
\label{def:Mxs}
\\
\mathfrak{M}(\varsigma, T)
= \mathfrak{M}(\varsigma, T; \boldsymbol{\kappa}) 
& := \max_{\substack{\sigma\ge \varsigma\\ 1\le |\tau|\le T}} 
|\boldsymbol{\zeta}(\boldsymbol{\kappa}s)|^2.
\label{def:frakMvarsigmaT}
\end{align}
Let $A'$ be a fix large integer, and put
\begin{equation}\label{def:Nj}
N_j := \big(A'(\log T)^5 \, \mathfrak{M}(4\sigma_j-3/\kappa_1, 8T)\big)^{1/2(1/\kappa_1-\sigma_j)}.
\end{equation}

\vskip 2mm

We divide $\Delta_{j, k}$ into two classes $(W)$ and $(Y)$ as follows.

\vskip 1,5mm

$\bullet$
\textit{The case of $\sigma_j\le (1-\varepsilon)/\kappa_1$.}
\par
\vskip 1mm

$\Delta_{j, k}\in (W)$
if $\Delta_{j, k}$ contains at least one zero of 
$\boldsymbol{\zeta}(\boldsymbol{\kappa}s)\boldsymbol{L}(\boldsymbol{\kappa}s, \boldsymbol{\chi})$, 
and  $\Delta_{j, k}\in (Y)$ otherwise.

\vskip 1,5mm

$\bullet$
\textit{The case of $(1-\varepsilon)/\kappa_1<\sigma_j\le (1-\delta_T)/\kappa_1$.}
\par
\vskip 1mm

We write
\begin{equation}\label{UB1/2}
\Delta_{j, k}\in (W)
\;\Leftrightarrow\;
\text{$\exists \; s\in \Delta_{j, k}$ such that}\;\,
|\boldsymbol{\zeta}(\boldsymbol{\kappa}s)
\boldsymbol{L}(\boldsymbol{\kappa}s, \boldsymbol{\chi})M_{N_j}(s)|<\tfrac{1}{2}
\end{equation}
and
\begin{equation}\label{LB1/2}
\Delta_{j, k}\in (Y)
\;\Leftrightarrow\;
\big|\boldsymbol{\zeta}(\boldsymbol{\kappa}s)
\boldsymbol{L}(\boldsymbol{\kappa}s, \boldsymbol{\chi})M_{N_j}(s)\big|\ge \tfrac{1}{2}
\;\;
\text{for all $s\in \Delta_{j, k}$.}
\end{equation}
 
\vskip 1mm

For each $k$, we define $j_k := \max\{j : \Delta_{j, k}\in (W)\}$ and put
\begin{equation}\label{defD'D0}
{\mathscr D}'
:= \mathop{\cup}_{0\le k\le K_T} \mathop{\cup}_{0\le j\le j_k}\Delta_{j, k}
\qquad\text{and}\qquad
{\mathscr D}_0
:= \mathop{\cup}_{0\le k\le K_T} \mathop{\cup}_{j_k<j\le J_T}\Delta_{j, k}.
\end{equation}
Clearly ${\mathscr D}_0$ consists of $\Delta_{j, k}$ of class $(Y)$ only.

\vskip 1mm

The Hooley-Huxley-Motohashi contour 
${\mathscr L}_T = {\mathscr L}_T(\varepsilon, \boldsymbol{\kappa}, \boldsymbol{\chi})$ 
is sysmmetric about the real axe.
Its supérieur part is the path in ${\mathscr D}_0$ consisting of horizontal and vertical line segments whose distances away from  ${\mathscr D}'$ are respectively  $d_{\rm h}$ and $d_{\rm v}$,
where $d_{\rm h}$ and $d_{\rm v}$ are defined by
\begin{equation}\label{defdh_dv}
\begin{aligned}
d_{\rm h}
& := (\log_2T)/\kappa_1,
\\
d_{\rm v}
& := \begin{cases}
\varepsilon^2/\kappa_1 & \text{if $\,\sigma\le (1-\varepsilon)/\kappa_1$},
\\
(\log T)^{-1}/\kappa_1   & \text{if $\,(1-\varepsilon)/\kappa_1<\sigma<(1-\delta_T)/\kappa_1$}.
\end{cases}
\end{aligned}
\end{equation}
Clearly ${\mathscr L}_T$ is symmetric about the real axis.
The following figure shows its upper part 
[from the point $((1-\varepsilon)/\kappa_1, 0)$ to the point $((1-\delta_T)/\kappa_1, T)$].

\begin{center}
\begin{tikzpicture}[font=\tiny]


\draw[dotted] (1,0)--(1,20);
\draw[dotted] (1.5,0)--(1.5,20);
\draw[dotted] (2,0)--(2,20);
\draw[dotted] (2.5,0)--(2.5,20);
\draw[dotted] (3,0)--(3,20);
\draw[dotted] (3.5,0)--(3.5,20);
\draw[dotted] (4,0)--(4,20);
\draw[dotted] (4.5,0)--(4.5,20);
\draw[dotted] (5,0)--(5,20);
\draw[dotted] (5.5,0)--(5.5,20);
\draw[dotted] (6,0)--(6,20);
\draw[dotted] (6.5,0)--(6.5,20);
\draw[dotted] (7,0)--(7,20);
\draw[dotted] (7.5,0)--(7.5,20);
\draw[dotted] (8,0)--(8,20);

\draw[dotted] (0.5,4)--(8.5,4);
\draw[dotted] (0.5,8)--(8.5,8);
\draw[dotted] (0.5,12)--(8.5,12);
\draw[dotted] (0.5,16)--(8.5,16);
\draw[dotted] (0.5,20)--(8.5,20);
\fill (-2,20) circle(1pt);
\node at (-2.25,20) {$T$};

\draw (2.75,10) -- ++(-0.4,0.4);
\node at (2.15,10.55) {$\Delta_{j,k}$};

\draw[dashed] (2.5,8) --(2.5,12);
\draw[dashed] (3,8) --(3,12);
\fill (2.5,0) circle(1pt);
\node at (2.5,-0.25) {$\sigma_j$};
\fill (3,0) circle(1pt);
\node at (3,-0.25) {$\sigma_{j+1}$};

\draw[dashed] (2.5,8) --(3,8);
\draw[dashed] (2.5,12) --(3,12);
\fill (-2,8) circle(1pt);
\node at (-2.4,8) {$\tau_k$};
\fill (-2,12) circle(1pt);
\node at (-2.4,12) {$\tau_{k+1}$};

\draw (4.5,0) -- ++(0,4) -- ++(0.5,0) -- ++(0,-4) --++(-0.5,0);
\draw (4.75,1.75) -- ++(-0.4,0.4);
\node at (4.05,2.3) {$\Delta_{j_0,0}$};

\draw (6.5,4) -- ++(0,4) -- ++(0.5,0) -- ++(0,-4) --++(-0.5,0);
\draw (5,8) -- ++(0,4) -- ++(0.5,0) -- ++(0,-4) --++(-0.5,0);
\draw (5.25,10) -- ++(-0.4,0.4);
\node at (4.6,10.55) {$\Delta_{j_k,k}$};
\draw (3,16) -- ++(0,4) -- ++(0.5,0) -- ++(0,-4) --++(-0.5,0);
\draw (3.25,18) -- ++(-0.4,0.4);
\node at (2.5,18.6) {$\Delta_{j_K,K}$};

\draw (4.25,20) -- (8.5,20);
\draw[line width=1pt] (4.25,20) --++(0,-5)
--++(-3,0) --++(0,-2) --++(5,0) --++(0,-4) --++(1.25,0)
--++(0,-6) --++(-1.75,0) --++(0,-3);
\draw (8.5,0) --++(0.2,0) arc (180:0:0.55);
\draw (4.3,16) -- ++(0.6,0.6);
\node at (5,16.75) {$\mathfrak{M}_{T}$};

\draw (9.25,0) --++(-0.38,0.38);
\draw (8.87,0.38) -- ++(0.04,0); 
\draw (8.87,0.38) -- ++(0,-0.04);
\node at (9.25,0.275) {$r$};

\draw (5.5,10) -- ++(0.75,0);
\draw (5.5,10) -- ++(0.05,0.05); 
\draw (5.5,10) -- ++(0.05,-0.05);
\draw (6.25,10) -- ++(-0.05,-0.05);
\draw (6.25,10) -- ++(-0.05,0.05);
\draw (5.875,10) -- ++(0.7,0.7); 
\node at (6.75,10.85) {$d_\text{v}$};

\draw (5.25,12) -- ++(0,1);
\draw (5.25,12) -- ++(0.05,0.05);
\draw (5.25,12) -- ++(-0.05,0.05);
\draw (5.25,13) -- ++(-0.05,-0.05);
\draw (5.25,13) -- ++(0.05,-0.05);
\draw (5.25,12.5) -- ++(0.9,0.9); 
\node at (6.25,13.55) {$d_{\rm h}$};

\node at (1.25,18) {${\mathscr D}'$};
\node at (7.25,18) {${\mathscr D}_0$};

\draw (8.5,0) -- (8.5,20);
\draw (8.5,20) -- (10,20);
\fill (8.5,0) circle(1pt);
\draw (8.5,0) -- ++(-0.4,-0.4);
\draw (8.1,-0.6) node {$(1-\delta_T)/\kappa_1$};

\draw[dashed] (5.875,0)--(5.875,20);
\fill (5.875,0) circle(1pt);
\node at (5.875, -0.25) {$(1-\varepsilon)/\kappa_1$};

\draw[dashed] (0.5,0)--(0.5,20);
\fill (0.5,0) circle(1pt);
\node at (0.5, -0.25) {$1/2\kappa_1$};

\draw (10,0)--(10,20);
\fill (10,0) circle(1pt);
\node at (10, -0.25) {$b$};

\fill (9.25,0) circle(1pt);
\node at (9.25, -0.25) {$1/\kappa_1$};

\begin{scope}[>=latex]
\draw[->] (-2.5,0) --(10.6,0) node[below] {$\sigma$};
\draw[->] (-2,-0.4) --(-2,20.5) node [left] {$\tau$};
\fill (-2,0) circle(1pt);
\node at (-2.2, -0.2) {O};
\end{scope}

\end{tikzpicture}
\end{center}
\begin{center}
{\small Figure 1 -- Superieur part of the contour ${\mathscr L}_T$}
\end{center}


\section{Lower and upper bounds of 
$|\boldsymbol{\zeta}(\boldsymbol{\kappa}s)\boldsymbol{L}(\boldsymbol{\kappa}s, \boldsymbol{\chi})|$ 
on ${\mathscr L}_T$}

\vskip 1mm

The aim of this section is to prove the following proposition.

\begin{proposition}\label{pro3.1}
Under the previous notation, we have
\begin{equation}\label{pro2.1Eq.A}
T^{-544\sqrt{2\varepsilon}(1-\kappa_1\sigma)}(\log T)^{-4}
\ll |\boldsymbol{\zeta}(\boldsymbol{\kappa}s)\boldsymbol{L}(\boldsymbol{\kappa}s, \boldsymbol{\chi})|
\ll T^{136\sqrt{2\varepsilon}(1-\kappa_1\sigma)}(\log T)^{4}.
\end{equation}
for all $s\in {\mathscr L}_T$, 
where the implied constants depend on $(\varepsilon, \boldsymbol{\kappa}, \boldsymbol{\chi})$.
\end{proposition}

First we establish two preliminary lemmas.

\begin{lemma}\label{lem3.2}
Under the previous notation, we have
\begin{equation}\label{EqLem2.1A}
{\rm e}^{-(\log T)^{1-\varepsilon^2}}
\ll |\boldsymbol{\zeta}(\boldsymbol{\kappa}s)\boldsymbol{L}(\boldsymbol{\kappa}s, \boldsymbol{\chi})|
\ll {\rm e}^{(\log T)^{1-\varepsilon^2}}
\end{equation}
for $s\in {\mathscr L}_T$ with $\sigma\leq (1-\varepsilon)/\kappa_1$, 
or $s$ with $(1-\varepsilon)/\kappa_1<\sigma\le \tfrac{1}{\kappa_1}(1-\varepsilon+\varepsilon^2)$
on the horizontal segments in ${\mathscr L}_T$ that intersect the vertical line $\re s=(1-\varepsilon)/\kappa_1$.
Here the implied constant depends only on $\varepsilon$.
\end{lemma}

\begin{proof}
Let $s=\sigma+\text{i}\tau$ satisfy the conditions in this lemma.
Without loss of generality, we can suppose that $\tau\ge T_0(\varepsilon, \boldsymbol{\kappa}, \boldsymbol{\chi})$.
Let us consider the four circles ${\mathscr C}_1$, ${\mathscr C}_2$, ${\mathscr C}_3$ and ${\mathscr C}_4$,
all centered at $s_0 := \log_2\tau+\text{i}\tau$,
with radii
\begin{align*}
r_1
& := \log_2\tau-(1+\eta)/\kappa_1,
\\
r_2
& := \log_2\tau-\sigma,
\\
r_3
& := \log_2\tau-\sigma+\varepsilon^2/(2\kappa_1),
\\
r_4
& := \log_2\tau-\sigma+\varepsilon^2/\kappa_1,
\end{align*}
respectively.
Here $\eta>0$ is a parameter to be chosen later.
We note that these four circles pass through the points
$(1+\eta)/\kappa_1+\text{i}\tau$,
$\sigma+\text{i}\tau$,
$\sigma-\varepsilon^2/(2\kappa_1)+\text{i}\tau$
and
$\sigma-\varepsilon^2/\kappa_1+\text{i}\tau$,
respectively.

Clearly $\boldsymbol{\zeta}(\boldsymbol{\kappa}s)
\boldsymbol{L}(\boldsymbol{\kappa}s, \boldsymbol{\chi})\not=0$ in a region containing the disc $|s-s_0|\le r_4$.
Thus we can unambiguously define 
$\log(\boldsymbol{\zeta}(\boldsymbol{\kappa}s)
\boldsymbol{L}(\boldsymbol{\kappa}s, \boldsymbol{\chi}))$ in this region.
We fix a branch of the logarithm throughout the remaining discussion.

Let $M_i$ denote the maximum of $|\log(\boldsymbol{\zeta}(\boldsymbol{\kappa}s)
\boldsymbol{L}(\boldsymbol{\kappa}s, \boldsymbol{\chi}))|$ on ${\mathscr C}_i$ 
relative to this branch.
By Hadamard's three-circle theorem and the fact that $s=\sigma+\text{i}\tau$ is on ${\mathscr C}_2$, we have
\begin{equation}\label{(2.14)}
|\log(\boldsymbol{\zeta}(\boldsymbol{\kappa}s)
\boldsymbol{L}(\boldsymbol{\kappa}s, \boldsymbol{\chi}))| \leq M_2 \leq M_1^{1-\phi}M_3^{\phi},
\end{equation}
where
\begin{align*}
\phi
& = \frac{\log(r_2/r_1)}{\log(r_3/r_1)}
\\
& =\frac{\log(1+(1+\eta-\kappa_1\sigma)/(\kappa_1\log_2\tau-1-\eta))}
{\log(1+(1+\eta-\kappa_1\sigma+\varepsilon^2/2)/(\kappa_1\log_2\tau-1-\eta))}
\\
& = \frac{1+\eta-\kappa_1\sigma}{1+\eta-\kappa_1\sigma+\varepsilon^2/2}
+ O\Big(\frac{1}{\log_2\tau}\Big).
\end{align*}
On taking 
$\eta=\kappa_1\sigma-\frac{1}{2}\big(1+\varepsilon^2+\frac{\varepsilon^2}{1+\varepsilon^2}\big)$
($\eta \geq \frac{\varepsilon^4}{2(1+\varepsilon^2)}$, since $\sigma\ge \frac{1}{\kappa_1}(\tfrac{1}{2}+\varepsilon^2)$), 
we have
\begin{equation}\label{(2.15)}
\phi=1-\varepsilon^2-\varepsilon^4+O((\log_2\tau)^{-1}).
\end{equation}

On the circle ${\mathscr C}_1$, we have
\begin{equation}\label{(2.16)}
M_1
\le \max_{\re s \geq (1+\eta)/\kappa_1}
\sum_{n=2}^{\infty} \left|\frac{\Lambda(n)}{n^{\kappa_1s}\log n}\right|
\leq \sum_{n=2}^{\infty}\frac1{n^{1+\eta}}\ll \frac1{\eta},
\end{equation}
where $\Lambda(n)$ is the von Mangoldt function.

In order to bound $M_3$,
we shall apply the Borel-Carath\'eodory theorem to the function $\log(\boldsymbol{\zeta}(\boldsymbol{\kappa}s)
\boldsymbol{L}(\boldsymbol{\kappa}s, \boldsymbol{\chi}))$
on the circles  ${\mathscr C}_3$, ${\mathscr C}_4$.
On the circle ${\mathscr C}_4$, it is well known that
$$
\re\big(\log(\boldsymbol{\zeta}(\boldsymbol{\kappa}s)
\boldsymbol{L}(\boldsymbol{\kappa}s, \boldsymbol{\chi}))\big)
= \log|\boldsymbol{\zeta}(\boldsymbol{\kappa}s)
\boldsymbol{L}(\boldsymbol{\kappa}s, \boldsymbol{\chi})| 
\ll \log \tau
$$
thanks to the convexity bounds of $\zeta(s)$ and of $L(s, \chi)$.
Hence the Borel-Carath\'eodory theorem gives 
\begin{equation}\label{(2.17)}
\begin{aligned}
M_3
& \le \frac{2r_3}{r_4-r_3} \max_{|s-s_0|\le r_4} \log|\boldsymbol{\zeta}(\boldsymbol{\kappa}s)
\boldsymbol{L}(\boldsymbol{\kappa}s, \boldsymbol{\chi})| 
+\frac{r_4+r_3}{r_4-r_3} 
|\log(\boldsymbol{\zeta}(\boldsymbol{\kappa}s_0)\boldsymbol{L}(\boldsymbol{\kappa}s_0, \boldsymbol{\chi}))|
\\
& \ll \frac{2(\log_2\tau-\sigma+\varepsilon^2/(2\kappa_1))}{\varepsilon^2/(2\kappa_1)} \log\tau
+ \frac{2(\log_2\tau-\sigma+\varepsilon^2/(2\kappa_1))}{\varepsilon^2/(2\kappa_1)} 
\\\noalign{\vskip 1mm}
& \ll (\log_2\tau)\log \tau.
\end{aligned}
\end{equation}

From \eqref{(2.14)}, \eqref{(2.15)}, \eqref{(2.16)} and \eqref{(2.17)}, we deduce that
\begin{align*}
|\log(\boldsymbol{\zeta}(\boldsymbol{\kappa}s)\boldsymbol{L}(\boldsymbol{\kappa}s, \boldsymbol{\chi}))|
& \ll (\eta^{-1})^{1-\phi} ((\log_2\tau)\log\tau)^{\phi}
\\
& \ll_{\varepsilon} ((\log_2\tau)\log\tau)^{1-\varepsilon^2-\varepsilon^4}
\\
& \ll_{\varepsilon} (\log\tau)^{1-\varepsilon^2} .
\end{align*}
This leads to the required estimates.
\end{proof}

\begin{lemma}\label{lem3.3}
Under the previous notation, we have
\begin{equation}\label{Eq:lem3.3A}
T^{-544r(1-\kappa_1\sigma_j)^{3/2}} (\log T)^{-c_0r}
\ll |\boldsymbol{\zeta}(\boldsymbol{\kappa}s)\boldsymbol{L}(\boldsymbol{\kappa}s, \boldsymbol{\chi})|
\ll T^{136r(1-\kappa_1\sigma_j)^{3/2}} (\log T)^{c_0r}
\end{equation}
for $s\in {\mathscr L}_T$ with $(1-\varepsilon)/\kappa_1<\sigma_j\le \sigma<\sigma_{j+1}$.
Here the implied constants depend on $(\boldsymbol{\kappa}, \boldsymbol{\chi})$.
In particular we have
\begin{equation}\label{Eq:lem3.3B}
T^{-544r\sqrt{2\varepsilon}(1-\kappa_1\sigma_j)} (\log T)^{-c_0r}
\ll |\boldsymbol{\zeta}(\boldsymbol{\kappa}s)\boldsymbol{L}(\boldsymbol{\kappa}s, \boldsymbol{\chi})|
\ll T^{136r\sqrt{2\varepsilon}(1-\kappa_1\sigma_j)} (\log T)^{c_0r}
\end{equation}
for $s\in {\mathscr L}_T$ with $(1-\varepsilon)/\kappa_1<\sigma_j\le \sigma<\sigma_{j+1}$.
All the implied constants are absolute.
\end{lemma}

\begin{proof}
By \eqref{Richert} with $w=1$ and Lemma \ref{lem2.1}, we have
\begin{equation}\label{Eqlem2.2D}
|\boldsymbol{\zeta}(\boldsymbol{\kappa}s)\boldsymbol{L}(\boldsymbol{\kappa}s, \boldsymbol{\chi})|
\ll \tau^{100r(1-\kappa_1\sigma)^{3/2}}(\log \tau)^{c_0r}
\qquad
(\tfrac{1}{2\kappa_1}\le \sigma\le \tfrac{1}{\kappa_1}, \, \tau\ge 2).
\end{equation}
This immediately implies the second inequality in \eqref{Eq:lem3.3A}.

Next we consider the first inequality in \eqref{Eq:lem3.3A}.
Let $s\in {\mathscr L}_T$ with $(1-\varepsilon)/\kappa_1<\sigma_j\le \sigma<\sigma_{j+1}$.
Since $1/\kappa_1\ge 4\sigma_j-3/\kappa_1\ge (1-4\varepsilon)/\kappa_1\le 1/2\kappa_1$,
the inequality \eqref{Richert} with $w=1$ allows us to derive that
\begin{equation}\label{UBNj}
\begin{aligned}
N_j^{2(1/\kappa_1-\sigma_j)}
& = A'(\log T)^5 
\max_{\substack{\sigma\ge 4\sigma_j-3/\kappa_1\\ 1\le |\tau|\le 8T}} |\boldsymbol{\zeta}(\boldsymbol{\kappa}s)|^2
\\
& \ll T^{200\sqrt{2}r(1-\kappa_1\sigma_j)^{3/2}}(\log T)^{c_0r}.
\end{aligned}
\end{equation}
According to the definition of ${\mathscr L}_T$,
there is an integer $k$ such that $s\in \Delta_{j, k}$ and
this $\Delta_{j, k}$ must be in $(Y)$ and \eqref{LB1/2} holds for all $s$ of this $\Delta_{j, k}$.
On the other hand, \eqref{(tau23)*(tau23*)} and \eqref{UBNj} imply that for $\sigma_j\le \sigma<\sigma_{j+1}$,
\begin{align*}
|M_{N_j}(s)|
& \le \sum_{n\le N_j} \tau_{\boldsymbol{\kappa}, \boldsymbol{\kappa}}(n) n^{-\sigma_j}
\ll (1-\kappa_1\sigma_j)^{-2} N_j^{1/\kappa_1-\sigma_j}
\\
& \ll T^{544(1-\kappa_1\sigma_j)^{3/2}}(\log T)^{c_0r}.
\end{align*}
Combining this with \eqref{LB1/2} immediately yields
$$
|\boldsymbol{\zeta}(\boldsymbol{\kappa}s)\boldsymbol{L}(\boldsymbol{\kappa}s, \boldsymbol{\chi})|
\ge (2|M_{N_j}(s)|)^{-1}
\gg T^{-544(1-\kappa_1\sigma_j)^{3/2}}(\log T)^{-c_0r}
$$
for $s\in {\mathscr L}_T$ with $(1-\varepsilon)/\kappa_1<\sigma_j\le \sigma<\sigma_{j+1}$.

Finally we note \eqref{Eq:lem3.3B} is a simple consequence of \eqref{Eq:lem3.3A} since
$(1-\varepsilon)/\kappa_1<\sigma_j$ implies that $(1-\kappa_1\sigma_j)^{1/2}\le \sqrt{\varepsilon}$.
\end{proof}

Now we are ready to prove Proposition \ref{pro3.1}.

\begin{proof}
When $s\in {\mathscr L}_T$ with $|\tau|\le 1$, the estimations \eqref{pro2.1Eq.A} are trivial.

Next suppose that $s\in {\mathscr L}_T$ with $|\tau|>1$. Then there is a $j$ 
such that $\sigma_j\le \sigma<\sigma_{j+1}$.
We consider the three possibilities.

\par
\vskip 1mm
$\bullet$
\textit{The case of $(1-\varepsilon)/\kappa_1<\sigma_j$.}
\par
\vskip 0,5mm

The inequality \eqref{pro2.1Eq.A} follows immediately from \eqref{Eq:lem3.3B} of Lemma \ref{lem3.3}.

\par
\vskip 1mm
$\bullet$
\textit{The case of $\sigma_j\le \sigma\le (1-\varepsilon)/\kappa_1$.}
\par
\vskip 0,5mm
In this case, the first part of Lemma \ref{lem3.2} shows that \eqref{pro2.1Eq.A} holds again since
$\sqrt{\varepsilon}(1-\kappa_j\sigma)\ge \varepsilon^{3/2}\ge (\log T)^{-\varepsilon^2}$
for $T\ge T_0(\varepsilon, \boldsymbol{\kappa}, \boldsymbol{\chi})$.

\par
\vskip 1mm
$\bullet$
\textit{The case of $\sigma_j\le (1-\varepsilon)/\kappa_1<\sigma$.}
\par
\vskip 0,5mm

In this case, $s$ must be on the horizontal segment in ${\mathscr L}_T$,
because the vertical segment keeps the distance $\varepsilon^2$ from the line $\re s=\sigma_j$
and $\sigma_j<\sigma<\sigma_{j+1}$.
Thus we can apply the second part of Lemma \ref{lem3.2} to get \eqref{pro2.1Eq.A} as before.
\end{proof}

\vskip 8mm

\section{Montgomery's method}

\smallskip

\subsection{Montgomery's mean value theorem}\

\vskip 1mm

Let
$\tau_{\boldsymbol{\kappa}, \boldsymbol{\kappa}}(n)$
and
$\mathfrak{M}(\varsigma, T)$
be defined as in \eqref{def:taunkappakappa}
and
\eqref{def:frakMvarsigmaT}, respectively.
The following proposition is a variant of \cite[Theorem 8.4]{Montgomery1971} for our purpose,
which will play a key role in the proof of Proposition \ref{pro5.1} below.

\begin{proposition}\label{prop4.1}
Let $\sigma_0>0, \varrho>0$ and $T\ge 2$ three real numbers.
Let $r\in \N$ and $\boldsymbol{\kappa}$ be as before. 
Let $\mathcal{S} = \mathcal{S}_T(\sigma_0, \varrho)$ be a finite set of complex numbers $s=\sigma+\rm{i}\tau$
such that
\begin{equation}\label{S_Condition1}
\sigma\ge \sigma_0,
\qquad
1\le |\tau|\le T
\end{equation}
for all $s\in \mathcal{S}$ and
\begin{equation}\label{S_Condition2}
|\tau-\tau'|\ge \varrho
\end{equation}
for any distinct points $s=\sigma+\rm{i}\tau$ and $s'=\sigma'+\rm{i}\tau'$ in $\mathcal{S}$.
For any sequence of complex numbers  $\{a_n\}$ verifying
\begin{equation}\label{Conditionan}
a_n\not=0
\;\Rightarrow\;
\tau_{\boldsymbol{\kappa}, \boldsymbol{\kappa}}(n)\ge 1,
\end{equation}
real numbers $\theta\in (1/\kappa_2, 1/\kappa_1)$ and $N\ge 1$,
we have
\begin{equation}\label{prop3.1.Eq.A}
\begin{aligned}
& \sum_{s\in \mathcal{S}} \Big|\sum_{n\le N} a_n n^{-s}\Big|^2
\\
& \ll \bigg\{(1+\varrho^{-1})N^{1/\kappa_1} 
+ \bigg(\mathfrak{M}(\theta, 4T)+\frac{1}{(1-\kappa_1\theta)(\kappa_2\theta-1)}\bigg) N^{\theta} |\mathcal{S}|\bigg\} 
\sum_{n\le N} \frac{|a_n|^2}{n^{2\sigma_0}},
\end{aligned}
\end{equation}
where the implied constant depends on $\boldsymbol{\kappa}$ only.
\end{proposition}

First we prove two preliminary lemmas.
The first one is due to Bombieri (see also \cite[Lemma 5.1]{Montgomery1971}).
For the convenience of reader, we give a direct proof.

\begin{lemma}\label{lem4.2}
Let $\mathcal{S}$ be a finite set of complex numbers $s$.
For any $\{a_n\}_{1\le n\le N}\subset \C$, we have
\begin{equation}\label{lem3.1.Eq.A}
\sum_{s\in \mathcal{S}} \Big|\sum_{n\le N} a_n n^{-s}\Big|^2
\le \Big(\sum_{n\le N} |a_n|^2 b_n^{-1}\Big)
\max_{s\in \mathcal{S}} \sum_{s'\in \mathcal{S}} |B(\overline{s}+s')|
\end{equation}
where 
$$
B(s) := \sum_{n\ge 1} b_n n^{-s}
$$
is absolutely convergent at the points $\overline{s}+s'$,
and the $b_n$ are non-negative real numbers for which $b_n>0$ whenever $a_n\not=0$.  
Here the implied constant is absolute.
\end{lemma}

\begin{proof}
According our hypothesis on $a_n$ and $b_n$, we can write $a_n = a_n b_n^{-1/2} b_n^{1/2}$ for all $n\le N$.
Introducing the notation
$$
A(s) := \sum_{n\le N} a_n n^{-s},
$$
then we can write
\begin{align*}
\sum_{s\in \mathcal{S}} |A(s)|^2
& = \sum_{s\in \mathcal{S}} 
\sum_{n\le N} a_n b_n^{-1/2} b_n^{1/2} n^{-s} \sum_{m\le N} \overline{a_m} m^{-\overline{s}}
\\
& = \sum_{n\le N} (a_n b_n^{-1/2}) 
\Big(b_n^{1/2} \sum_{s\in \mathcal{S}} n^{-s} \sum_{m\le N} \overline{a_m} m^{-\overline{s}}\Big).
\end{align*}
By the Cauchy inequality, it follows that
\begin{equation}\label{lem3.1.Eq.B}
\begin{aligned}
\Big(\sum_{s\in \mathcal{S}} |A(s)|^2\Big)^2
& \le \Big(\sum_{n\le N} |a_n|^2 b_n^{-1}\Big) \Upsilon,
\end{aligned}
\end{equation}
where
$$
\Upsilon
:= \sum_{n\le N} b_n \Big|\sum_{s\in \mathcal{S}} n^{-s} \sum_{m\le N} \overline{a_m} m^{-\overline{s}}\Big|^2.
$$
Since $B(s)$ is absolutely convergent at the points $\overline{s}+s'$, we can deduce
\begin{align*}
\Upsilon
& \le \sum_{n\ge 1} b_n \Big|\sum_{s\in \mathcal{S}} n^{-s} \sum_{m\le N} \overline{a_m} m^{-\overline{s}}\Big|^2
\\
& = \sum_{n\ge 1} b_n
\sum_{s'\in \mathcal{S}} n^{-s'} \sum_{d\le N} \overline{a_d} d^{-\overline{s'}}
\sum_{s\in \mathcal{S}} n^{-\overline{s}} \sum_{m\le N} a_m m^{-s}
\\
& = \sum_{s\in \mathcal{S}} \sum_{s'\in \mathcal{S}} 
\sum_{d\le N} \overline{a_d} d^{-\overline{s'}} \sum_{m\le N} a_m m^{-s}
B(\overline{s}+s').
\end{align*}
On the other hand, the trivial inequality $|ab|\le \tfrac{1}{2}(|a|^2+|b|^2)$ allows us to write
$$
\Big|\sum_{d\le N} \overline{a_d} d^{-\overline{s'}} \sum_{m\le N}a_m m^{-s}\Big|
\le \frac{1}{2}\Big(|A(s)|^2 + |A(s')|^2\Big).
$$
Thus
\begin{align*}
\Upsilon
& \le \sum_{s\in \mathcal{S}} \sum_{s'\in \mathcal{S}} 
\frac{1}{2}\Big(|A(s)|^2 + |A(s')|^2\Big) |B(\overline{s}+s')|
\\
& \le \frac{1}{2}\sum_{s\in \mathcal{S}} \sum_{s'\in \mathcal{S}} 
|A(s)|^2 |B(\overline{s}+s')|
+ \frac{1}{2}\sum_{s\in \mathcal{S}} \sum_{s'\in \mathcal{S}} 
|A(s')|^2 |B(\overline{s}+s')|
\\
& \le \frac{1}{2}\sum_{s\in \mathcal{S}} 
|A(s)|^2 \max_{s\in \mathcal{S}} \sum_{s'\in \mathcal{S}} |B(\overline{s}+s')|
+ \frac{1}{2} \sum_{s'\in \mathcal{S}} |A(s')|^2 
\max_{s'\in \mathcal{S}} \sum_{s\in \mathcal{S}} |B(\overline{s}+s')|.
\end{align*}
Noticing that 
$$
|B(\overline{s}+s')| 
= |\overline{B(\overline{s}+s')}|
= |B(\overline{\overline{s}+s'})|
= |B(s+\overline{s'})|,
$$
the precede inequality becomes
\begin{align*}
\Upsilon
\le \sum_{s\in \mathcal{S}} 
|A(s)|^2 \max_{s\in \mathcal{S}} \sum_{s'\in \mathcal{S}} |B(\overline{s}+s')|.
\end{align*}
Inserting it into \eqref{lem3.1.Eq.B}, we get the required result.
\end{proof}

The next lemma is an analogue of \cite[page 157, Theorem II.2]{Montgomery1971}.

\begin{lemma}\label{lem4.3}
Let $r$, $\boldsymbol{\kappa}$, $\tau_{\boldsymbol{\kappa}, \boldsymbol{\kappa}}(n)$ 
and $\mathfrak{M}(\theta, T)$be as before.
For any $s = \sigma + {\rm i}\tau$ with $\sigma\ge 0$ and $|\tau|\ge 1$, 
real numbers $\theta\in (1/\kappa_2, 1/\kappa_1)$ and $N\ge 1$,
we have
\begin{equation}\label{lem3.2.Eq.A}
\begin{aligned}
& \bigg|\sum_{n\ge 1} 
\frac{\tau_{\boldsymbol{\kappa}, \boldsymbol{\kappa}}(n)}{n^s} \big({\rm e}^{-n/(2N)}-{\rm e}^{-n/N}\big)\bigg|
\\
& \ll N^{1/\kappa_1} {\rm e}^{-|\tau|} + \big\{\mathfrak{M}(\theta, 2|\tau|) + (1-\kappa_1\theta)^{-1}(\kappa_2\theta-1)^{-1}\big\} N^{\theta},
\end{aligned}
\end{equation}
where the implied constant depends on $\boldsymbol{\kappa}$ only.
\end{lemma}

\begin{proof}
Denote by $\mathscr{S}_N(s) = \mathscr{S}_N(s; \boldsymbol{\kappa}, \boldsymbol{\chi})$ 
the series on the left-hand side of \eqref{lem3.2.Eq.A}.
By the Perron formula \cite[page 151, Lemma]{Titchmarsh1986}, we can write
\begin{equation}\label{lem3.2.Eq.B}
\mathscr{S}_N(s)
= \frac{1}{2\pi\text{i}}
\int_{2/\kappa_1-\text{i}\infty}^{2/\kappa_1+\text{i}\infty} 
\boldsymbol{\zeta}(\boldsymbol{\kappa}(w+s))^2 \Gamma(w) \big((2N)^{w}-N^{w}\big) \d w,
\end{equation}
where $\boldsymbol{\zeta}(\boldsymbol{\kappa}s)$ 
is defined as in \eqref{def:tauan} above.
We take the contour to the line $\re w = \theta-\sigma$ 
with $1/\kappa_2-\sigma<\theta-\sigma<1/\kappa_1-\sigma$, 
and in doing so we pass a simple pole at $w=1/\kappa_1-s$.
Put
$$
G(w) := \prod_{2\le j\le r} \zeta(\kappa_j(w+s))^2 \Gamma(w) \big((2N)^{w}-N^{w}\big).
$$
The residue of the integrand at this pole is, in view of the hypothesis $\sigma\ge 0$,
\begin{align*}
\displaystyle\mathop{\text{Res}}_{w=1/\kappa_1-s}
\boldsymbol{\zeta}(\boldsymbol{\kappa}(w+s))^2 \Gamma(w) \big((2N)^{w}-N^{w}\big)
& = \big(\gamma G(1/\kappa_1-s)+G'(1/\kappa_1-s)/\kappa_1\big)/\kappa_1
\\\noalign{\vskip -0,5mm}
& \ll N^{1/\kappa_1} \text{e}^{-|\tau|},
\end{align*}
where we have used the Stirling formula \cite[page 151]{Titchmarsh1952}:
\begin{equation}\label{Stirling}
|\Gamma(s)|
= \sqrt{2\pi} \, \text{e}^{-(\pi/2)|\tau|} |\tau|^{\sigma-1/2}
\bigg\{1 + O\bigg(\frac{|\tan(\frac{\arg s}{2})|}{|\tau|}
+ \frac{|a|^2+|b|^2}{|\tau|^2}
+ \frac{|a|^3+|b|^3}{|\tau|^3}\bigg)\bigg\}
\end{equation}
valable uniformly for $a, b\in \R$ with $a<b$, $a\le \sigma\le b$ and $|\tau|\ge 1$,
where the implied $O$-constant is absolute.

On the other hand, by using the Stirling formula again 
it is easy to see that for $\re w=\theta-\sigma$ and $|\im w|\le |\tau|$, the integrand in \eqref{lem3.2.Eq.B} is
$$
\ll \big\{\mathfrak{M}(\theta, 2|\tau|) + (1-\kappa_1\theta)^{-1}(\kappa_2\theta-1)^{-1}\big\} 
N^{\theta-\sigma} {\rm e}^{-|\im w|},
$$
while otherwise the Stirling formula and the convexity bounds of $\zeta$-function imply that
$$
\ll N^{\theta-\sigma} {\rm e}^{-|\im w|}
\ll \big(\mathfrak{M}(\theta, 2|\tau|) + (1-\kappa_1\theta)^{-1}(\kappa_2\theta-1)^{-1}\big) 
N^{\theta-\sigma} {\rm e}^{-|\im w|}.
$$
The required result follows from this.
\end{proof}

Now we are ready to prove Proposition \ref{prop4.1}.
\par
\begin{proof}
We shall apply Lemma \ref{lem4.2} with the choice of 
$$
b_n := \tau_{\boldsymbol{\kappa}, \boldsymbol{\kappa}}(n) n^{2\sigma_0} \big(\text{e}^{-n/(2N)}-\text{e}^{-n/N}\big).
$$ 
In view of the simple fact that $\text{e}^{-n/(2N)}-\text{e}^{-n/N}\asymp 1$ for $N\le n\le 2N$
and the hypothesis \eqref{Conditionan}, we have
\begin{equation}\label{prop3.1.Eq.B}
\sum_{s\in \mathcal{S}} \Big|\sum_{N\le n\le 2N} a_n n^{-s}\Big|^2
\ll \Big(\sum_{N\le n\le 2N} |a_n|^2 n^{-2\sigma_0}\Big)
\max_{s\in \mathcal{S}} \sum_{s'\in \mathcal{S}} |B(\overline{s}+s')|,
\end{equation}
where 
$$
B(s) := \sum_{n\ge 1} 
\frac{\tau_{\boldsymbol{\kappa}, \boldsymbol{\kappa}}(n)}{n^{s-2\sigma_0}} \big({\rm e}^{-n/(2N)}-{\rm e}^{-n/N}\big).
$$

In view of \eqref{S_Condition1} and \eqref{S_Condition2}, we can apply Lemma \ref{lem4.3} to deduce
\begin{align*}
|B(\overline{s}+s')|
& = \bigg|\sum_{n\ge 1} 
\frac{\tau_{\boldsymbol{\kappa}, \boldsymbol{\kappa}}(n)}{n^{\sigma + \sigma' - 2\sigma_0+\text{i}(\tau' - \tau)}}
\big(\text{e}^{-n/(2N)}-\text{e}^{-n/N}\big)\bigg|
\\
& \ll N^{1/\kappa_1} {\rm e}^{-|\tau' - \tau|}
+ \big\{\mathfrak{M}(\theta, 2|\tau' - \tau|) + (1-\kappa_1\theta)^{-1}(\kappa_2\theta-1)^{-1}\big\} N^{\theta}
\end{align*}
for any distinct points $s=\sigma+{\rm i}\tau\in \mathcal{S}$ and $s'=\sigma'+{\rm i}\tau'\in \mathcal{S}$.
By \eqref{S_Condition2}, the contribution of the term $N^{1/\kappa_1} {\rm e}^{|\tau' - \tau|}$ to the sum
$\sum_{\substack{s'\in \mathcal{S}\\ s'\not=s}} |B(\overline{s}+s')|$ is
$$
\ll N^{1/\kappa_1} \sum_{\substack{s'\in \mathcal{S}\\ |\tau-\tau'|\ge \varrho}} {\rm e}^{-|\tau-\tau'|}
\ll N^{1/\kappa_1} \sum_{n\ge 0} {\rm e}^{-n} \sum_{\substack{s'\in \mathcal{S}\\ n<|\tau-\tau'|\le n+1}} 1
\ll (1+\varrho^{-1})N^{1/\kappa_1}.
$$
Thus
\begin{equation}\label{prop3.1.Eq.C}
\begin{aligned}
& \sum_{s'\in \mathcal{S}, \, s'\not=s} |B(\overline{s}+s')|
\\
& \ll (1+\varrho^{-1}) N^{1/\kappa_1}
+ \big\{\mathfrak{M}(\theta, 2|\tau' - \tau|) + (1-\kappa_1\theta)^{-1}(\kappa_2\theta-1)^{-1}\big\} N^{\theta} |\mathcal{S}|.
\end{aligned}
\end{equation}
When $s=s'$, we have
\begin{equation}\label{prop3.1.Eq.D}
B(\overline{s}+s') 
= B(2\sigma)
\le \sum_{n\ge 1} \tau_{\boldsymbol{\kappa}, \boldsymbol{\kappa}}(n) \big(\text{e}^{-n/(2N)}-\text{e}^{-n/N}\big)
\ll N^{1/\kappa_1}.
\end{equation}

Now the required result follows from \eqref{prop3.1.Eq.B}, \eqref{prop3.1.Eq.C} and \eqref{prop3.1.Eq.D}.
\end{proof}

\vskip 8mm

\section{Density estimation of small value points}

In \cite{Montgomery1971}, Montgomery developed a new method for studying zero-densities of the Riemann $\zeta$-function
and of the Dirichlet $L$-functions.
Subsequently by modifying this method, Huxley \cite{Huxley1972} established his zero-density estimation \eqref{Huxley12/5}.
In \cite{Motohashi1976}, Motohashi noted that Montgomery's method can be adapted
to estimate the density of  ``\textit{small value points}'' of $\zeta(s)$
(see \cite[Section 2.3]{CuiLvWu2018} for a detail description).
Here we shall adapt this method to prove the following proposition.

\begin{proposition}\label{pro5.1}
Under the previous notation, for $j = 0, 1, \dots, J_T$ we have
\begin{equation}\label{EqLem2.3A}
\big|\big\{k\le K_T : \Delta_{j, k}\in (W)\big\}\big|
\ll T^{\psi(1-\kappa_1\sigma_j)}(\log T)^{\eta} 
\end{equation}
if $\kappa_1\sigma_j\le 1-\varepsilon$, and
\begin{equation}\label{EqLem2.3B}
\big|\big\{k\le K_T : \Delta_{j, k}\in (W)\big\}\big|
\ll T^{500r(1-\kappa_1\sigma_j)^{3/2}}(\log T)^{10r} 
\end{equation}
if $1-\varepsilon\le \kappa_1\sigma_j\le 1-\delta_T$.
Here the implied constants depend on $\boldsymbol{\kappa}$ and $\varepsilon$ only.
\end{proposition}

\begin{proof}
When $\kappa_1\sigma_j\le 1-\varepsilon$,
the number of $\Delta_{j, k}$ of type $(W)$ does not exceed the number of non-trivial zeros of 
$\boldsymbol{\zeta}(\boldsymbol{\kappa}s) \boldsymbol{L}(\boldsymbol{\kappa}s, \boldsymbol{\chi})$.
Thus
$$
|\{k\le K_T : \Delta_{j, k}\in (W)\}|
\le \sum_{1\le i\le r} \big(N(\kappa_i\sigma_j, 2T) + N_{\chi_i}(\kappa_i\sigma_j, 3T)\big).
$$
Now the required bound follows from \eqref{ZeroDensity}.

Next we suppose $1-\varepsilon\le \kappa_1\sigma_j\le 1-\delta_T$.
\par
\vskip 0,5mm

Let ${\mathcal K}_j(T)$ be a subset of the set $\{\log T\le k\le K_T : \Delta_{j, k}\in (W)\}$
such that the difference of two distinct integers of ${\mathcal K}_j(T)$ is at least $3A'$,
where $A'$ is the large integer specified in \eqref{def:Mxs}.
Obviously
$$
|\{(\log T)^{2}\le k\le K_T : \Delta_{j, k}\in (W)\}|
\le 3A' |{\mathcal K}_j(T)|.
$$
Therefore it suffices to show that
\begin{equation}\label{EqLem2.3C}
|{\mathcal K}_j(T)|
\ll_{\varepsilon} T^{170(1-\kappa_1\sigma_j)^{3/2}}(\log T)^{13}
\end{equation}
for $T\ge T_0(\varepsilon, \boldsymbol{\kappa}, \boldsymbol{\chi})$.

Let $M_x(s)$ be defined as in \eqref{def:Mxs} and write
\begin{equation}\label{EqLem2.3D}
\Phi_x(s) 
:= \boldsymbol{\zeta}(\boldsymbol{\kappa}s) \boldsymbol{L}(\boldsymbol{\kappa}s, \boldsymbol{\chi}) M_x(s).
\end{equation}
Let $\phi_{x}(n)$ be the $n$th coefficient of the Dirichlet series $\Phi_x(s)$, then
\begin{equation}\label{EqLem2.3E}
\phi_{x}(n) = \sum_{d\mid n, \, d\le x} \tau_{\boldsymbol{\kappa}}^{\langle-1\rangle}(d; \boldsymbol{\chi}) 
\tau_{\boldsymbol{\kappa}}(n/d; \boldsymbol{\chi}).
\end{equation}

By the Perron formula \cite[Lemma, page 151]{Titchmarsh1986}, we can write
$$
\sum_{n\ge 1} \frac{\phi_{x}(n)}{n^s} \text{e}^{-n/y}
= \frac{1}{2\pi\text{i}}
\int_{2/\kappa_1-\text{i}\infty}^{2/\kappa_1+\text{i}\infty} \Phi_x(w+s) \Gamma(w) y^{w} \d w
$$
for $y>x\ge 3$ and $s = \sigma+\text{i}\tau\in \C$ with $\sigma_j\le \sigma<\sigma_{j+1}$.
We take the contour to the line $\re w = \alpha_j-\sigma<0$ 
with $\alpha_j := 4\sigma_j-3/\kappa_1<\sigma_j<1/\kappa_1$,
and in doing so we pass two simple poles at $w=0$ and $w=1/\kappa_1-s$.
Our equation becomes
$$
\sum_{n\ge 1} \frac{\phi_{x}(n)}{n^s} \text{e}^{-n/y}
= \Phi_x(s) + \Psi_{x, y}(s) + I(s; x, y),
$$
where
\begin{align*}
\Psi_{x, y}(s)
& := \prod_{2\le j\le r} \zeta(\kappa_j/\kappa_1)
\boldsymbol{L}(\boldsymbol{\kappa}/\kappa_1, \boldsymbol{\chi})
M_x(1/\kappa_1) \Gamma(1/\kappa_1-s) y^{1/\kappa_1-s},
\\
I(s; x, y)
& := \frac{1}{2\pi} \int_{-\infty}^{+\infty} \Phi_x(\alpha_j+\text{i}\tau+\text{i}u)
\Gamma(\alpha_j-\sigma+\text{i}u) y^{\alpha_j-\sigma+\text{i}u} \d u.
\end{align*}

In view of \eqref{EqLem2.3E} and \eqref{(tau23)*(tau23*)}, 
we have 
\begin{equation}\label{EqLem2.3E2}
\phi_{x}(n) = \begin{cases}
1 & \text{if $\,n=1$,}
\\
0 & \text{if $\,1<n\le x$,}
\end{cases}
\end{equation}
and
\begin{equation}\label{EqLem2.3E2'}
|\phi_{x}(n)|
\le (\tau_{\boldsymbol{\kappa}, \boldsymbol{\kappa}}*
\tau_{\boldsymbol{\kappa}, \boldsymbol{\kappa}})(n)
\quad
(n>x),
\end{equation}
where $\tau_{\boldsymbol{\kappa}, \boldsymbol{\kappa}}(n)$ is defined as in \eqref{def:taunkappakappa} above.
It is easy to see that 
$$
\sum_{n\le t} (\tau_{\boldsymbol{\kappa}, \boldsymbol{\kappa}}*
\tau_{\boldsymbol{\kappa}, \boldsymbol{\kappa}})(n)
\ll t^{1/\kappa_1}(\log t)^3.
$$ 
By a simple partial integration, we can deduce that
\begin{align*}
\Big|\sum_{n>y^2} \frac{\phi_{x}(n)}{n^s} \text{e}^{-n/y}\Big|
& \le \int_{y^2}^{\infty} t^{-\sigma} \text{e}^{-t/y} \d \Big(\sum_{n\le t} 
(\tau_{\boldsymbol{\kappa}, \boldsymbol{\kappa}}*
\tau_{\boldsymbol{\kappa}, \boldsymbol{\kappa}})(n)\Big)
\\
& \ll \text{e}^{-y} y^{1-2\sigma} (\log y)^3
+ y^{-1} \int_{y^2}^{\infty} \text{e}^{-t/y}  t^{1/\kappa_1-\sigma} (\log t)^3 \d t
\\\noalign{\vskip 1mm}
& \ll \text{e}^{-y/2}
\end{align*}
for $\sigma\ge 1/(2\kappa_1)$.
Inserting it  into the precedent relation, we  find that
\begin{equation}\label{EqLem2.3F}
\text{e}^{-1/y} + \sum_{x<n\le y^2} \frac{\phi_{x}(n)}{n^s} \text{e}^{-n/y} + O(\text{e}^{-y/2})
= \Phi_x(s) + \Psi_{x, y}(s)
+ I(s; x, y)
\end{equation}
for $s\in \C$ with $\sigma_j\le \sigma<\sigma_{j+1}$ and $y>x\ge 3$.

If $k\in {\mathcal K}_j(T)$, then there is at least a $s_k := v_k + \text{i}t_k\in \Delta_{j, k}$ such that
\begin{equation}\label{EqLem2.3G}
|\Phi_{N_j}(s_k)|
= |\boldsymbol{\zeta}(\boldsymbol{\kappa}s_k) 
\boldsymbol{L}(\boldsymbol{\kappa}s_k; \boldsymbol{\chi}) M_{N_j}(s_k)|
\le \tfrac{1}{2}\cdot
\end{equation}
By the definition of ${\mathcal K}_j(T)$, we have
$$
\sigma_j\le v_k\le \sigma_{j+1},
\quad
(\log T)^2\le t_k\le T
\quad\,\text{and}\quad\;
|t_{k_1}-t_{k_2}|\ge 3A'\log T
\quad
(k_1\not=k_2).
$$

Since $|t_k|\ge (\log T)^2$, the Stirling formula \eqref{Stirling} allows us to deduce
\begin{equation}\label{EqLem2.3H}
\begin{aligned}
|\Psi_{x, y}(s_k)|
& = \Big|\prod_{2\le j\le r} \zeta(\kappa_j/\kappa_1)
\boldsymbol{L}(\boldsymbol{\kappa}/\kappa_1, \boldsymbol{\chi})
M_x(1/\kappa_1) \Gamma(1/\kappa_1-s_k) y^{1/\kappa_1-s_k}\Big|
\\\noalign{\vskip -0,5mm}
& \ll (\log x) y^{1/2-v_k} \text{e}^{-(\pi/2)|t_k|} |t_k|^{1/2-v_k}
\\\noalign{\vskip 1,5mm}
& \le \tfrac{1}{10}
\end{aligned}
\end{equation}
for all $3\le x\le y\le T^{100}$.

Similarly, using the estimates
\begin{align*}
& \boldsymbol{\zeta}(\boldsymbol{\kappa}(\alpha_j+\text{i}t_k+\text{i}u))
\ll (T + |u|)^r,
\\
& \boldsymbol{L}(\boldsymbol{\kappa}(\alpha_j+\text{i}t_k+\text{i}u), \boldsymbol{\chi})
\ll (T + |u|)^r,
\\
& M_x(\alpha_j+\text{i}t_k+\text{i}u)
\ll x^{1/\kappa_1-\alpha_j}\log x\ll T^{100}
\end{align*}
and the Stirling formula \eqref{Stirling},
we derive that
\begin{equation}\label{EqLem2.3I}
\int_{|u|\ge A'\log T} \big|\Phi_x(\alpha_j+\text{i}t_k+\text{i}u)
\Gamma(\alpha_j-v_k+\text{i}u)\big| y^{\alpha_j-v_k} \d u
\le \tfrac{1}{10}
\end{equation}
for all $3\le x\le y\le T^{100}$.

Taking $(s, x) = (s_k, N_j)$ in \eqref{EqLem2.3F}
and combining  with \eqref{EqLem2.3G}, \eqref{EqLem2.3H} and \eqref{EqLem2.3I},
we easily see that
\begin{equation}\label{EqLem2.3J}
\bigg|\sum_{N_j<n\le y^2} \frac{\phi_{N_j}(n)}{n^{s_k}} \text{e}^{-n/y}\bigg|
\ge \frac{1}{6}
\end{equation}
or
\begin{equation}\label{EqLem2.3K}
\bigg|\int_{-A'\log T}^{A'\log T} \Phi_{N_j}(\alpha_j+\text{i}t_k+\text{i}u)
\Gamma(\alpha_j-v_k+\text{i}u) y^{\alpha-v_k+\text{i}u} \d u\bigg|
\ge \frac{1}{6}
\end{equation}
or both.

Let ${\mathcal K}_j'(T)$ and ${\mathcal K}_j''(T)$ be the subsets of ${\mathcal K}_j(T)$
for which \eqref{EqLem2.3J} and \eqref{EqLem2.3K} hold  respectively.
Then
\begin{equation}\label{EqLem2.3M}
|{\mathcal K}_j(T)|\le |{\mathcal K}_j'(T)| + |{\mathcal K}_j''(T)|.
\end{equation}

First we bound $|{\mathcal K}_j'(T)|$.
By a dyadic argument, there is a $U\in [N_j, y^2]$ such that
\begin{equation}\label{EqLem2.3N}
\bigg|\sum_{U<n\le 2U} \frac{\phi_{N_j}(n)}{n^{s_k}} \text{e}^{-n/y}\bigg|
\ge \frac{1}{18\log y}
\end{equation}
holds for $\gg |{\mathcal K}_j'(T)|(\log y)^{-1}$ integers $k\in {\mathcal K}_j'(T)$.
Let ${\mathcal S}'$ be the set of corresponding points $s_k$.
With the help of \eqref{EqLem2.3E2},
it is easy to see that 
$\phi_{N_j}(n)\not=0\,\Rightarrow\,\tau_{\boldsymbol{\kappa}, \boldsymbol{\kappa}}(n)\ge 1$.
Thus we can apply Proposition \ref{prop4.1} with $a_n=\phi_{N_j}(n)$ and 
$\theta = \alpha_j := 4\sigma_j-3/\kappa_1$.
In view of the bound
\begin{align*}
\sum_{U<n\le 2U} 
\frac{(\tau_{\boldsymbol{\kappa}, \boldsymbol{\kappa}}*\tau_{\boldsymbol{\kappa}, \boldsymbol{\kappa}})(n)^2}
{n^{2\sigma_j}} \text{e}^{-2n/y}
& \ll \text{e}^{-2U/y} \int_{U}^{2U} t^{-2\sigma_j} 
\d \Big(\sum_{U<n\le t} (\tau_{\boldsymbol{\kappa}, \boldsymbol{\kappa}}*\tau_{\boldsymbol{\kappa}, \boldsymbol{\kappa}})(n)^2\Big)
\\
& \ll U^{1/\kappa_1-2\sigma_j} (\log T)^3 \text{e}^{-2U/y},
\end{align*}
it follows that
\begin{equation}\label{EqLem2.3O}
\begin{aligned}
& \sum_{s_k\in {\mathcal S}'} \bigg|\sum_{U<n\le 2U} \frac{\phi_{N_j}(n)}{n^{s_k}} \text{e}^{-n/y}\bigg|^2
\\
& \ll \Big(U^{2(1/\kappa_1-\sigma_j)} 
+ |{\mathcal S}'| U^{-2(1/\kappa_1-\sigma_j)} \mathfrak{M}(\alpha_j, 4T)
\Big) \text{e}^{-2U/y} (\log T)^3.
\end{aligned}
\end{equation}
Since $U\ge N_j$, we have
$$
U^{-2(1/\kappa_1-\sigma_j)} (\log T)^3 \, \mathfrak{M}(\alpha_j, 4T)
\le A'^{-1} (\log T)^{-2}.
$$
On the other hand,
the inequality \eqref{EqLem2.3N} implies that the member on the left-hand side of \eqref{EqLem2.3O} is
$$
\ge |{\mathcal S}'| (18\log y)^{-2}
\ge |{\mathcal S}'| (1800\log T)^{-2}.
$$
Since $A'$ is a fixed large integer, the last term on the right-hand side of \eqref{EqLem2.3O}
is smaller than this lower bound. 
Thus it can be simplified as
$$
|{\mathcal S}'| (\log T)^{-2}
\ll U^{2(1/\kappa_1-\sigma_j)} (\log T)^3 \text{e}^{-2U/y}
$$
for all $N_j\le y\le T^{100}$ and some $U\in [N_j, y^2]$.
Noticing that
$$
|{\mathcal S}'|\gg |{\mathcal K}_j'(T)| (\log T)^{-1},
$$
we obtain
\begin{equation}\label{EqLem2.3P}
\begin{aligned}
|{\mathcal K}_j'(T)|
& \ll y^{2(1/\kappa_1-\sigma_j)} (\log T)^6
\qquad
(\text{for all $N_j\le y\le T^{100}$})
\\\noalign{\vskip 1mm}
& \ll N_j^{(10/3)(1/\kappa_1-\sigma_j)} (\log T)^{7/3}
\qquad
(\text{for $y$ given by \eqref{defy}}).
\end{aligned}
\end{equation}

\vskip 1mm

Next we bound $|{\mathcal K}_j''(T)|$.
Let $u_k\in [-A'\log T, A'\log T]$ such that
$$
\Phi_{N_J}(s_k') = \max_{|u|\le A\log T} |\Phi_{N_j}(\alpha_j+\text{i}(t_k+u))|
$$
where $s_k' := \alpha_j+\text{i}t_k'$ and $t_k' := t_k + u_k$.
Thus from \eqref{EqLem2.3K} we deduce that
\begin{align*}
\frac{1}{6}
& \le \bigg|\int_{-A'\log T}^{A'\log T} \Phi_{N_j}(\alpha_j+\text{i}(t_k+u))
\Gamma(\alpha_j-v_k+\text{i}u) y^{\alpha_j-v_k+\text{i}u} \d u\bigg|
\\
& \le y^{\alpha_j-v_k}
\big|\Phi_{N_j}(s_k')\big|
\int_{-A'\log T}^{A'\log T} \big|\Gamma(\alpha_j-v_k+\text{i}u)\big| \d u.
\end{align*}
Since $\Gamma(s)$ has a simple pole at $s=0$ and $|\alpha_j-v_k|\gg (\log T)^{-1}$,
we can derive, via \eqref{Stirling}, that
$$
\int_{-A'\log T}^{A'\log T} \big|\Gamma(\alpha_j-v_k+\text{i}u)\big| \d u\ll \log T
$$
and thus
$$
1 \ll y^{\alpha_j-\sigma_j} \big|M_{N_j}(s_k')\big| \mathfrak{M}(\alpha_j, 8T) \log T,
$$
or equivalently
$$
\big|M_{N_j}(s_k')\big|
\gg y^{\sigma_j-\alpha} 
\big(\mathfrak{M}(\alpha_j, 8T)\log T\big)^{-1}.
$$
Hence there is a $V\in [1, N_j]$ such that
$$
\Big|\sum_{V<n\le 2V} \tau_{\boldsymbol{\kappa}}^{\langle-1\rangle}(n; \boldsymbol{\chi}) n^{-s_k'}\Big|
\gg y^{\sigma_j-\alpha} \mathfrak{M}(\alpha_j, 8T)^{-1} (\log T)^{-2}
$$
holds for $\gg |{\mathcal K}_j''(T)|(\log T)^{-1}$ integers $k\in {\mathcal K}_j''(T)$.
Let ${\mathcal S}''$ be the corresponding set of points $s_k'$.
We note $|t_k'|\le 2T$ and
$$
|t_{k_1}'-t_{k_2}'|\ge |t_{k_1}-t_{k_2}|-|u_{k_1}-u_{k_2}|\ge A'\log T.
$$
Using Proposition \ref{prop4.1} with $\theta=\alpha_j:=4\sigma_j-3/\kappa_1$ 
and $a_n=\tau_{\boldsymbol{\kappa}}^{\langle-1\rangle}(n; \boldsymbol{\chi})$
and in view of the bound
$$
\sum_{V<n\le 2V} \tau_{\boldsymbol{\kappa}, \boldsymbol{\kappa}}(n)^2 n^{-2\alpha_j}
\ll V^{1/\kappa_1-2\alpha_j} (\log V)^3
\ll V^{7/\kappa_1-8\sigma_j} (\log V)^3,
$$
it follows that
\begin{equation}\label{EqLem2.3Q}
\sum_{s_k'\in {\mathcal S}''} \Big|\sum_{V<n\le 2V} 
\tau_{\boldsymbol{\kappa}}^{\langle-1\rangle}(n; \boldsymbol{\chi}) n^{-s_k'}\Big|^2
\ll \big(V^{8(1/\kappa_1-\sigma_j)} + |{\mathcal S}''| \mathfrak{M}(\alpha_j, 8T) V^{4(1/\kappa_1-\sigma_j)}\big) 
(\log V)^3.
\end{equation}
Take $y$ such that
\begin{equation}\label{defy}
\begin{aligned}
y^{2(\sigma_j-\alpha_j)}
& = {A'}^3 N_j^{4(1/\kappa_1-\sigma_j)} \mathfrak{M}(\alpha_j, 8T)^{3} (\log T)^{4}
\\
& = N_j^{10(1/\kappa_1-\sigma_j)} (\log T)^{-11}.
\end{aligned}
\end{equation}
The left-hand side of \eqref{EqLem2.3Q} is
$$
\ge |{\mathcal S}''| y^{2(\sigma_j-\alpha_j)} \mathfrak{M}(\alpha_j, 8T)^{-2}
(\log T)^{-4}.
$$
Hence the inequality \eqref{EqLem2.3Q} can be simplified as
$$
|{\mathcal S}''| y^{2(\sigma_j-\alpha_j)} \mathfrak{M}(\alpha_j, 8T)^{-2}
(\log T)^{-4}
\ll N_j^{8(1/\kappa_1-\sigma_j)}.
$$
With
$$
|{\mathcal S}''|\gg  |{\mathcal K}_j''(T)|(\log T)^{-1},
$$
we  deduce that
\begin{equation}\label{EqLem2.3R}
\begin{aligned}
|{\mathcal K}_j''(T)|
& \ll N_j^{8(1/\kappa_1-\sigma_j)} y^{2(\alpha_j-\sigma_j)} \mathfrak{M}(\alpha_j, 8T)^{2} (\log T)^5
\\
& \ll N_j^{2(1/\kappa_1-\sigma_j)} (\log T)^7.
\end{aligned}
\end{equation}
On combining \eqref{EqLem2.3M}, \eqref{EqLem2.3P} and \eqref{EqLem2.3R}, it follows that
$$
|{\mathcal K}_j(T)|
\ll N_j^{(10/3)(1/\kappa_1-\sigma_j)} (\log T)^3.
$$
Now the required inequality follows from \eqref{UBNj}.
This completes the proof.
\end{proof}

\vskip 8mm

\section{Proof of Theorem \ref{thm1.1}}

\vskip 1mm

We shall conserve the notation in Section 2. First we prove a lemma.

\begin{lemma}\label{lem3.1}
Let $r\in \N$, 
$\boldsymbol{\kappa} := (\kappa_1, \dots, \kappa_r)\in \N^r$ with $1\le \kappa_1<\dots <\kappa_r\le 2\kappa_1$,
$\boldsymbol{z} := (z_1, \dots, z_r)\in \C^r$, 
$\boldsymbol{B} := (B_1, \dots, B_r)\in (\R^{+*})^r$,
and let $\alpha>0$, $\delta\ge 0$, $A\ge 0$, $M>0$
be some constants.
Suppose that the Dirichlet series
$$
{\mathcal F}(s) := \sum_{n=1}^{\infty} f(n)n^{-s}
$$
is of type ${\mathcal P}(\boldsymbol{\kappa}, \boldsymbol{z}, \boldsymbol{w}, 
\boldsymbol{B}, \boldsymbol{C}, \alpha, \delta, A, M)$.
Then there is an absolute positive constant $D$  and a constant $B = B_1 + \cdots + B_r + C_1 + \cdots + C_r$ 
such that we have
\begin{equation}\label{lem3.1Eq.1}
{\mathcal F}(s)
\ll MD^{B} T^{(\delta+B\sqrt{\varepsilon})(1-\kappa_1\sigma)} (\log T)^{A+B}
\end{equation}
for all $s\in {\mathscr L}_T$, where the implied constant depends only on $\varepsilon$.
\end{lemma}

\begin{proof}
Since we have chosen the principal value of complex logarithm,
we can write
\begin{align*}
\big|\boldsymbol{\zeta}(\boldsymbol{\kappa}s)^{\boldsymbol{z}}
\boldsymbol{L}(\boldsymbol{\kappa}s, \boldsymbol{\chi})^{\boldsymbol{w}}\big|
& = \prod_{1\le i\le r} |\zeta(\kappa_i s)|^{\Re e\, z_i} |L(\kappa_i s, \chi_i)|^{\Re e\, w_i}
\text{e}^{-(\Im m\, z_i) \arg \zeta(\kappa_i s)-(\Im m\, w_i) \arg L(\kappa_i s, \chi_i)}
\\
& \le \text{e}^{\pi (B_1 + \cdots + B_r + C_1 + \cdots + C_r)} 
\prod_{1\le i\le r} |\zeta(\kappa_i s)|^{\Re e\, z_i} |L(\kappa_i s, \chi_i)|^{\Re e\, w_i}
\end{align*}
for all $s\in \C$ verifying $\prod_{1\le i\le r} \zeta(\kappa_i s)L(\kappa_i s, \chi_i)\not=0$.
Invoking Proposition \ref{pro3.1}, 
we see that there is a suitable absolute constant $D$ and a constant $B = B(\boldsymbol{B}, \boldsymbol{C})$ 
depending on $(\boldsymbol{B}, \boldsymbol{C})$ such that
\begin{equation}\label{lem3.1Eq.3}
\big|\boldsymbol{\zeta}(\boldsymbol{\kappa}s)^{\boldsymbol{z}}
\boldsymbol{L}(\boldsymbol{\kappa}s, \boldsymbol{\chi})^{\boldsymbol{w}}\big|
\ll_{\varepsilon, \boldsymbol{\chi}} D^{B} T^{B\sqrt{\varepsilon}(1-\kappa_1\sigma)} (\log T)^{B}
\end{equation}
for all $s\in {\mathscr L}_T$,
where the implied constant depends only on $(\varepsilon, \boldsymbol{\chi})$.

Finally the required bound \eqref{lem3.1Eq.1} follows from \eqref{lem3.1Eq.3} and the hypothesis \eqref{UBGszw}.
\end{proof}

Now we are ready to prove Theorem \ref{thm1.1}.

Since the Dirichlet series ${\mathcal F}(s)$ is of type 
${\mathcal P}(\boldsymbol{\kappa}, \boldsymbol{z}, \boldsymbol{w}, 
\boldsymbol{B}, \boldsymbol{C}, \alpha, \delta, A, M)$,
we can apply \cite[Corollary II.2.2.1]{Tenenbaum1995}
with the choice of parameters
$\sigma_a=1/\kappa_1$,
$\alpha=\alpha$,
$\sigma=0$
to write
$$
\sum_{x<n\le x+x^{1-1/\kappa_1}y} f(n)
= \frac{1}{2\pi\text{i}} \int_{b-\text{i}T}^{b+\text{i}T} {\mathcal F}(s) \frac{(x+x^{1-1/\kappa_1}y)^s - x^s}{s} \d s
+ O_{\varepsilon}\bigg(M\frac{x^{1/\kappa_1+\varepsilon}}{T}\bigg),
$$
where $b = 1/\kappa_1 + 1/\log x$ and $\text{e}^{\sqrt{\log x}} \le T\le x$ is a parameter to be chosen later.

Denote by $\Gamma_T$ the path formed from the circle $|s-1/\kappa_1|=r_0:=1/(2\kappa_1\log x)$ 
excluding the point $s=1/\kappa_1-r_0$,
together with the segment $[(1-\delta_T)/\kappa_1, 1/\kappa_1-r_0]$ 
traced out twice with respective arguments $+\pi$ and $-\pi$.
By the residue theorem, the path $[b-\text{i}T, b+\text{i}T]$ is deformed  into
$
\Gamma_T\cup [(1-\delta_T)/\kappa_1-\text{i}T, \, (1-\delta_T)/\kappa_1+\text{i}T]
\cup [(1-\delta_T)/\kappa_1\pm \text{i}T, \, b\pm \text{i}T].
$
In view of Lemma \ref{lem3.1}, 
for any $a\in (1/(2\kappa_1), 1/\kappa_1)$, the integral over the horizontal segments $[a\pm \text{i}T, \, b\pm \text{i}T]$ is
\begin{align*}
& \int_{a\pm \text{i}T}^{b\pm \text{i}T} 
\bigg|{\mathcal F}(s) \frac{(x+x^{1-1/\kappa_1}y)^s - x^s}{s}\bigg| |\d s|
\\\noalign{\vskip 1mm}
& \ll \frac{M D^C (\log T)^{A+B}}{T}
\int_{a}^{b} T^{\max\{(\delta+B\sqrt{\varepsilon}) (1-\kappa_1\sigma), \, 0\}} x^{\sigma} \d\sigma
\\
& \ll M D^C \frac{x^{1/\kappa_1}}{T} (\log T)^{A+B}
\bigg(\int_{a}^{1/\kappa_1} 
\bigg(\frac{x^{1/\kappa_1}}{T^{\delta+B\sqrt{\varepsilon}}}\bigg)^{\kappa_1\sigma-1} \d\sigma
+ 1\bigg)
\\
& \ll M D^C \frac{x^{1/\kappa_1}}{T} (\log T)^{A+B},
\end{align*}
provided
\begin{equation}\label{T_Condition1}
T^{\delta+B\sqrt{\varepsilon}}\le x^{1/\kappa_1}.
\end{equation}
Thus
\begin{equation}\label{Section3.2}
\sum_{x<n\le x+x^{1-1/\kappa_1}y} f(n)
= I + O\bigg(MD^C\frac{x^{1/\kappa_1+\varepsilon}}{T}\bigg),
\end{equation}
where 
$$
I := \frac{1}{2\pi\text{i}} \int_{\Gamma_T\cup [(1-\delta_T)/\kappa_1-\text{i}T, \, (1-\delta_T)/\kappa_1+\text{i}T]} {\mathcal F}(s) \frac{(x+x^{1-1/\kappa_1}y)^s - x^s}{s} \d s.
$$
and the implied constant depends on $(\varepsilon, \boldsymbol{\chi})$ only.

Let ${\mathscr L}_T$ be the Motohashi contour defined as in Section \ref{section2}.
Consider the two symmetric simply connected regions bounded by ${\mathscr L}_T$,
the segment $[(1-\delta_T)/\kappa_1-\text{i}T, \, (1-\delta_T)/\kappa_1+\text{i}T]$
and the two line segments $[\sigma_{j_0+1}+d_{\rm v}, (1-\delta_T)/\kappa_1]$  with respective arguments $+\pi$ and $-\pi$ measured from the real axis on the right of $1-\delta_T$.
It is clear that ${\mathcal F}(s)$ is analytic in these two simply connected regions.
Denote by $\Gamma_T^*$ the path joining (the two end-points of) $\Gamma_T$  with the two line segments
$[\sigma_{j_0+1}+d_{\rm v}, (1-\delta_T)/\kappa_1]$ of the symmetric regions.
Thanks to the residue theorem, we can write
\begin{equation}\label{Section3.3}
I = I_1 + I_2,
\end{equation}
with
\begin{align*}
I_1
& := \frac{1}{2\pi\text{i}} \int_{\Gamma_T^*} {\mathcal F}(s) \frac{(x+x^{1-1/\kappa_1}y)^s - x^s}{s} \d s,
\\
I_2
& := \frac{1}{2\pi\text{i}} \int_{{\mathscr L}_T} {\mathcal F}(s) \frac{(x+x^{1-1/\kappa_1}y)^s - x^s}{s} \d s.
\end{align*}

\vskip 2mm

A.
\textit{Evaluation of $I_1$}

\vskip 1mm

According to our hypothesis, the function
$
s\mapsto Z(\kappa_1s; z_1)
\boldsymbol{\zeta}(\boldsymbol{\kappa}_*s)^{\boldsymbol{z}_*}
\boldsymbol{L}(\boldsymbol{\kappa}s, \boldsymbol{\chi})^{\boldsymbol{w}}
\mathcal{G}(s)
$
is holomorphic in the disc $|s-1/\kappa_1|<1/\kappa_1-1/\kappa_2$.
In view of \eqref{UBGzetaZ}, the Cauchy integral formula implies that
\begin{equation}\label{UBgkkappaw}
g_{\ell}(\boldsymbol{\kappa}, \boldsymbol{z}, \boldsymbol{w}, \boldsymbol{\chi})
\ll M c^{-\ell}
\qquad
(\ell\ge 0, \, |\boldsymbol{z}|\le |\boldsymbol{B}|, \, |\boldsymbol{w}|\le |\boldsymbol{C}|),
\end{equation}
where $g_{\ell}(\boldsymbol{\kappa}, \boldsymbol{z}, \boldsymbol{w}, \boldsymbol{\chi})$ is defined as in \eqref{defgkzw} and
$c := \tfrac{2}{3}(1/\kappa_1-1/\kappa_2)$. 
From this and \eqref{TaylorSeriesGzetaZ}, 
we deduce that for any integer $N\ge 0$ and 
$|s-1/\kappa_1|\le \tfrac{1}{2}(1/\kappa_1-1/\kappa_2)$,
$$
Z(\kappa_1s; z_1)
\boldsymbol{\zeta}(\boldsymbol{\kappa}_*s)^{\boldsymbol{z}_*}
\boldsymbol{L}(\boldsymbol{\kappa}s, \boldsymbol{\chi})^{\boldsymbol{w}}
\mathcal{G}(s)
= \sum_{\ell=0}^{N} g_{\ell}(\boldsymbol{\kappa}, \boldsymbol{z}, \boldsymbol{w}, \boldsymbol{\chi}) \big(s-\tfrac{1}{\kappa_1}\big)^{\ell} 
+ O\big(M(|s-\tfrac{1}{\kappa_1}|/c)^{N+1}\big).
$$
Thus we have
\begin{equation}\label{FormulaI}
I_1 = \sum_{\ell=0}^{N} \kappa_1^{-z_1} g_{\ell}(\boldsymbol{\kappa}, \boldsymbol{z}, \boldsymbol{w}, \boldsymbol{\chi}) M_{\ell}(x, y) + O\big(Mc^{-N} E_N(x, y)\big),
\end{equation}
where
\begin{align*}
M_{\ell}(x, y)
& := \frac{1}{2\pi\text{i}} \int_{\Gamma_T^*} (s-1/\kappa_1)^{\ell-z_1} \frac{(x+x^{1-1/\kappa_1}y)^s - x^s}{s} \d s,
\\
E_N(x, y)
& := \int_{\Gamma_T^*} \bigg|(s-1/\kappa_1)^{N+1-z_1}
 \frac{(x+x^{1-1/\kappa_1}y)^s - x^s}{s}\bigg| |\d s|.
\end{align*}

Firstly we evaluate $M_{\ell}(x, y)$.
Using the formula
\begin{equation}\label{Formula}
\frac{(x+x^{1-1/\kappa_1}y)^s - x^s}{s}
= \int_{x}^{x+x^{1-1/\kappa_1}y} t^{s-1} \d t
\end{equation}
and Corollary II.5.2.1 of \cite{Tenenbaum1995},
we  write
\begin{align*}
M_{\ell}(x, y)
& = \int_{x}^{x+x^{1-1/\kappa_1}y} \bigg(\frac{1}{2\pi\text{i}} 
\int_{\Gamma_T^*} (s-1/\kappa_1)^{\ell-z_1} t^{s-1}\d s\bigg) \d t
\\
& = \int_{x}^{x+x^{1-1/\kappa_1}y} t^{1/\kappa_1-1} (\log t)^{z_1-1-\ell} \bigg\{\frac{1}{\Gamma(\kappa_1-\ell)}
+ O\bigg(\frac{(c_1\ell+1)^{\ell}}{t^{\delta_T/2}}\bigg)\bigg\} \d t,
\end{align*}
where we have used the following inequality
$$
47^{|z_1-\ell|} \Gamma(1+|z_1-\ell|)
\ll_{B_1} (c_1\ell+1)^{\ell}
\quad
(\ell\ge 0, \, |z_1|\le B_1).
$$
The constant $c_1$ and the implied constant depend at most on $B_1$.
Besides  for $|z_1|\le B_1$, an elementary computation shows that
\begin{align*}
\int_{x}^{x+x^{1-1/\kappa_1}y} t^{1/\kappa_1-1}(\log t)^{z_1-1-\ell} \d t
& = \int_{0}^{x^{1-1/\kappa_1}y} (x+t)^{1/\kappa_1-1} (\log(x+t))^{z_1-1-\ell} \d t
\\
& = y' (\log x)^{z_1-1-\ell} \bigg\{1 + O_{B_1}\bigg(\frac{(\ell+1)y}{x^{1/\kappa_1}\log x}\bigg)\bigg\}.
\end{align*}
Inserting this into the preceeding formula, we obtain
\begin{equation}\label{EvaluationMkxy}
\begin{aligned}
M_{\ell}(x, y)
& = y' (\log x)^{z_1-1-\ell}
\bigg\{\frac{1}{\Gamma(z_1-\ell)} + O_{B_1}\bigg(\frac{(\ell+1)y}{\Gamma(z_1-\ell)x^{1/\kappa_1}\log x}
+ \frac{(c_1\ell+1)^{\ell}}{x^{\delta_T/2}}\bigg)\bigg\}
\end{aligned}
\end{equation}
for $\ell\ge 0$ and $|z_1|\le B_1$.

Next we estimate $E_N(x, y)$.
In view of the trivial inequality
\begin{equation}\label{TrivialInequality}
\bigg|\frac{(x+x^{1-1/\kappa_1}y)^s - x^s}{s}\bigg|\ll yx^{\sigma-1/\kappa_1},
\end{equation}
we deduce that
\begin{equation}\label{UBRNxy}
\begin{aligned}
E_N(x, y)
& \,\ll \int_{1/2\kappa_1+\varepsilon^2}^{1/\kappa_1-1/\log x} 
(1/\kappa_1-\sigma)^{N+1-\re z_1} x^{\sigma-1/\kappa_1}y \d \sigma
+ \frac{y}{(\log x)^{N+2-\re z_1}}
\\
& \,\ll \frac{y}{(\log x)^{N+2-\re z_1}}
\bigg(\int_{1}^{\infty} t^{N+1-\re z_1} \text{e}^{-t} \d t+ 1\bigg)
\\
& \,\ll y(\log x)^{\re z_1-1} \bigg(\frac{c_1N+1}{\log x}\bigg)^{N+1}
\end{aligned}
\end{equation}
uniformly for $x\ge y\ge 2$, $N\ge 0$ and $|z_1|\le B_1$,
where the constant $c_1>0$ and the implied constant depends only on $B_1$.

Inserting \eqref{EvaluationMkxy} and \eqref{UBRNxy} into \eqref{FormulaI}
and using \eqref{UBgkkappaw} and the fact that $y'\asymp y$,
we find that
\begin{equation}\label{EvaluationI1}
I_1 = y'(\log x)^{z-1}
\bigg\{\sum_{\ell=0}^{N} \frac{\lambda_{\ell}(\boldsymbol{\kappa}, \boldsymbol{z}, \boldsymbol{w}, \boldsymbol{\chi})}{(\log x)^{\ell}}
+ O_{B_1}\big(E_N^*(x, y)\big)\bigg\},
\end{equation}
where
\begin{align*}
E_N^*(x, y)
:= \frac{y}{x^{1/\kappa_1}}\sum_{\ell=1}^{N+1} 
\frac{\ell |\lambda_{\ell-1}(\boldsymbol{\kappa}, \boldsymbol{z}, \boldsymbol{w}, \boldsymbol{\chi})|}{(\log x)^{\ell}}
+ \frac{(c_1N+1)^{N+1}}{x^{\delta_T/2}}
+ M\bigg(\frac{c_1N+1}{\log x}\bigg)^{N+1}.
\end{align*}

\vskip 2mm

B.
\textit{Evaluation of $I_2$}

\vskip 1mm

Let $\mathscr{L}_T'$ be the union of those vertical line segments of $\mathscr{L}_T$ 
whose real part is equal to $\tfrac{1}{2\kappa_1}+\varepsilon^2$ and 
$\mathscr{L}_T'' := \mathscr{L}_T\sset \mathscr{L}_T'$.
Denote by $I_2'$ and $I_2''$ the contribution of $\mathscr{L}_T'$ and $\mathscr{L}_T''$ to $I_2$,
respectively.
Using the trivial inequality
$$
\bigg|\frac{(x+x^{1-1/\kappa_1}y)^s-x^s}{s}\bigg|
\ll \frac{x^{1/2\kappa_1+\varepsilon^2}}{|\tau|+1}
\qquad
(s\in \mathscr{L}_T')
$$
and Lemma \ref{lem3.1}, we can deduce
\begin{equation}\label{I2'}
\begin{aligned}
I_2'
& \ll M D^B x^{1/2\kappa_1+\varepsilon^2} 
T^{(\delta+B\sqrt{\varepsilon})(1/2-\kappa_1\varepsilon^2)}(\log T)^{A+4B+1}
\\
& \ll M x^{(1/2+\delta/2(\psi+\delta))/\kappa_1+\sqrt{\varepsilon}}
\\
& \ll M x^{(1-1/(\psi+\delta))/\kappa_1+\sqrt{\varepsilon}}
\end{aligned}
\end{equation}
with the value of $T$ given by \eqref{defT} below and $\psi\ge 2$.

Next we bound $I_2''$.
In view of \eqref{TrivialInequality}, we can write that
\begin{equation}\label{I2A}
\begin{aligned}
I_2''
& \ll y \int_{{\mathscr L}_T''} |\mathcal{F}(s)| x^{\sigma-1/\kappa_1} |\d s|
\\
& \ll y \sum_{0\le j\le J_T} \sum_{\substack{0\le k\le K_T\\ \Delta_{j, k}\in (W)}}
\int_{{\mathscr L}_T^{[j, k]}} |\mathcal{F}(s)| x^{\sigma-1/\kappa_1} |\d s|,
\end{aligned}
\end{equation}
where ${\mathscr L}_T^{[j, k]}$ is the vertical line segment of ${\mathscr L}_T''$ around $\Delta_{j, k}$ 
and the horizontal line segments with $\sigma\le \sigma_j+d_{\rm v}$.
Clearly the length of ${\mathscr L}_T^{[j, k]}$ is $\ll \log T$. 
Thus by Lemma \ref{lem3.1}, it follows 
$$
\int_{{\mathscr L}_T^{[j, k]}} |\mathcal{F}(s)| x^{\sigma-1} |\d s|
\ll M D^{B} (\log T)^{A+4B+1} T^{(\delta+B\sqrt{\varepsilon})(1-\kappa_1(\sigma_j+d_{\rm v}))} 
x^{\sigma_j+d_{\rm v}-1/\kappa_1}
$$
for all $0\le k\le K_T$.
Inserting it into \eqref{I2A} and using Proposition \ref{pro5.1}, 
we can deduce, with the notation $J_{T, 0}:=[(\tfrac{1}{2}-\varepsilon)\log T]$, that
$$
I_2''
\ll M D^{B} y (\log T)^{A+4B+18+\eta} (I_{2, *}'' + I_{2, \dagger}''),
$$
where
\begin{align*}
I_{2, *}''
& := \sum_{0\le j\le J_{T, 0}} T^{(\delta+B\sqrt{\varepsilon})(1-\kappa_1(\sigma_j+d_{\rm v}))} 
x^{\sigma_j+d_{\rm v}-1/\kappa_1}
\cdot T^{\psi(1-\kappa_1\sigma_j)},
\\
I_{2, \dagger}''
& := \sum_{J_{T, 0}<j\le J_T} T^{(\delta+B\sqrt{\varepsilon})(1-\kappa_1(\sigma_j+d_{\rm v}))} 
x^{\sigma_j+d_{\rm v}-1/\kappa_1}
\cdot
T^{100\sqrt{\varepsilon}(1-\sigma_j)}.
\end{align*}
Taking 
\begin{equation}\label{defT}
T := x^{(1-\kappa_1\sqrt{\varepsilon})/\kappa_1(\psi+\delta+B\sqrt{\varepsilon})}
\end{equation}
and in view of \eqref{defdh_dv},
it is easy to check that
$$
I_{2, *}''
\ll x^{\varepsilon^2/\kappa_1} \sum_{0\le j\le J_{T, 0}} 
\big(x^{1/\kappa_1}/T^{\psi+\delta+B\sqrt{\varepsilon}}\big)^{-(1-\kappa_1\sigma_j)} \log x
\ll x^{2\varepsilon^2/\kappa_1-\varepsilon^{3/2}} 
\ll x^{-\varepsilon^2}
$$
and
$$
I_{2, \dagger}''
\ll \sum_{J_{T, 0}<j\le J_T} \big(x/T^{\delta+100(B+1)\sqrt{\varepsilon}}\big)^{-(1-\kappa_1\sigma_j)}
\ll \text{e}^{-2c_2(\log x)^{1/3}(\log_2x)^{-1/3}}.
$$
Inserting it into the preceeding estimate for $I_2''$, we conclude that
\begin{equation}\label{UBI2''}
I_2''\ll_{\boldsymbol{B}} M y\text{e}^{-c_2(\log x)^{1/3}(\log_2x)^{-1/3}}.
\end{equation}
 
Now from \eqref{Section3.2}, \eqref{Section3.3}, \eqref{EvaluationI1}, \eqref{I2'} and \eqref{UBI2''}.
we deduce that
$$
\sum_{x<n\le x+x^{1-1/\kappa_1}y} f(n) 
= y'(\log x)^{z-1}
\bigg\{\sum_{\ell=0}^{N} \frac{\lambda_{\ell}(\boldsymbol{\kappa}, \boldsymbol{z}, \boldsymbol{w}, \boldsymbol{\chi})}{(\log x)^{\ell}}
+ O\big(R_N^*(x, y)\big)\bigg\}
$$
uniformly for
$x\ge 3$,
$x^{(1-1/(\psi+\delta))/\kappa_1+\varepsilon}\le y\le x^{1/\kappa_1}$,
$N\ge 0$,
$|\boldsymbol{z}|\le \boldsymbol{B}$ and
$|\boldsymbol{w}|\le \boldsymbol{C}$,
where 
$$
R_N^*(x, y)
:= \frac{y}{x^{1/\kappa_1}} \sum_{\ell=1}^{N+1} 
\frac{\ell |\lambda_{\ell-1}(\boldsymbol{\kappa}, \boldsymbol{z}, \boldsymbol{w}, \boldsymbol{\chi})|}{(\log x)^{\ell}}
+ M\bigg\{\bigg(\frac{c_1N+1}{\log x}\bigg)^{N+1} + \frac{(c_1N+1)^{N+1}}{{\rm e}^{c_2(\log x)^{1/3}(\log_2x)^{-1/3}}}\bigg\}
$$
for some constants $c_1>0$ and $c_2>0$ depending only on $\boldsymbol{B}$, $\boldsymbol{C}$, 
$\delta$ and $\varepsilon$.

It remains to prove that the first term on the right-hand side can be absorbed by the third.
In view of \eqref{UBGzetaZ}, the Cauchy formula allows us to write
$
g_{\ell}(\boldsymbol{\kappa}, \boldsymbol{z}, \boldsymbol{w}, \boldsymbol{\chi})
\ll_{A, \boldsymbol{B}, \boldsymbol{C}, \delta} M 3^{\ell}
$
for $|\boldsymbol{z}|\le \boldsymbol{B}$, $|\boldsymbol{w}|\le \boldsymbol{C}$ and $\ell\ge 1$. 
Combining this with the Stirling formula, we easily derive 
$\lambda_{\ell}(\boldsymbol{\kappa}, \boldsymbol{z}, \boldsymbol{w}, \boldsymbol{\chi})
\ll_{A, \boldsymbol{B}, \boldsymbol{C}, \delta} M (9/\ell)^{\ell}$
for $|\boldsymbol{z}|\le \boldsymbol{B}$, $|\boldsymbol{w}|\le \boldsymbol{C}$ and $\ell\ge 1$. 
This implies that
$$
\frac{y}{x^{1/\kappa_1}} \sum_{\ell=1}^{N+1} \frac{\ell |\lambda_{\ell-1}(\boldsymbol{\kappa}, \boldsymbol{z}, \boldsymbol{w}, \boldsymbol{\chi})|}{(\log x)^{\ell}}
\ll_{A, \boldsymbol{B}, \boldsymbol{C}, \delta} \frac{M y}{x^{1/\kappa_1}\log x}
$$
holds uniformly for 
$x\ge 3$,
$x^{(1-1/(\psi+\delta))/\kappa_1+\varepsilon}\le y\le x^{1/\kappa_1}$,
$N\ge 0$,
$|\boldsymbol{z}|\le \boldsymbol{B}$ and
$|\boldsymbol{w}|\le \boldsymbol{C}$.
This completes the proof.

\vskip 3mm

\noindent{\bf Acknowledgement}. 
This work is supported in part by NSFC (Grant Nos. 11771121 and 11471265).


\vskip 10mm

\begin{thebibliography}{CC}

\bibitem{Bateman_Grosswald1958}
P. T. Bateman and E. Grosswald,   
\textit{On a theorem of Erd\H os and Szekeres},   
Illinois J. Math. {\bf 2} (1958), 88--89.

\bibitem{CuiWu2014}
Z. Cui \& J. Wu,
\textit{The Selberg-Delange method in short intervals with an application},
Acta Arith. {\bf 163} (2014), no. 3, 247--260.

\bibitem{CuiLvWu2018}
Z. Cui, G.-S. L\"u \& J. Wu,
\textit{The Selberg-Delange method in short intervals with some applications},
Sci China Math, 2018, 61, https://doi.org/10.1007/s11425-017-9172-7.

\bibitem{Delange1959}
H. Delange,
\textit{Sur les formules dues \`a Atle Selberg},
Bull. Sc. Math. $2^{\circ}$ s\'erie {\bf 83} (1959), 101--111.

\bibitem{DDT1979}
J.-M. Deshouillers, F. Dress \& G. Tenenbaum, 
\textit{Lois de r\'epartition des diviseurs, 1},
Acta Arith. {\bf 23} (1979), 273--283. 

\bibitem{ErdosKac1939}
P. Erd\H os \& M. Kac,
\textit{On the Gaussian law of errors in the theory of additive functions},
Proc. Nat. Acad. Sci. U.S.A. {\bf 25} (1939), 206--207.

\bibitem{FengWu2018}
B. Feng \& J. Wu,
\textit{Beta law on divisors of integers representable as sum of two squares} (in Chinese),
Preprint, 2018.

\bibitem{Hooley1974}
C. Hooley,
\textit{On intervals between numbers that are sums of two squares III},
J. reine angew. Math. {\bf 267} (1974), 207--218.

\bibitem{Huxley1972}
M. N. Huxley,
\textit{The difference between consecutive primes},
Inven. Math. {\bf 15} (1972), 164--170.

\bibitem{Landau1909}
E. Landau,
\textit{Handbuch der Lehre von der Verteilung der Primzahlen} (2 vols.),
Teubner, Leipzig; 3rd edition: Chelsea, New York (1974).

\bibitem{Montgomery1969}
H. L. Montgomery,
\textit{Zeros of $L$-functions},
Inven. Math. {\bf 8} (1969), 346--354.

\bibitem{Montgomery1971}
H. L. Montgomery,
\textit{Topics in multiplicative number theory},
Lecture Notes in Mathematics {\bf 227},
Springer-Verlag, Berlin-Heidelberg-New York, 1971.

\bibitem{Motohashi1976}
Y. Motohashi,
\textit{On the sum of the M\"obis function in a short segment},
Proc. Japan Acad. {\bf 52} (1976), 477--479.

\bibitem{Ramachandra1976}
K. Ramachandra,
\textit{Some problems of analytic number theory},
Acta Arith. {\bf 31} (1976), 313--324.

\bibitem{Richert1967}
H. E. Richert,
\textit{Zur abschatzung der Riemannschen zetafunktion in der n\"ahe der vertikalen $\sigma=1$},
Math. Ann. {\bf 169} (1967), 97--101.

\bibitem{Sathe1953}
L. G. Sathe,
\textit{On a problem of Hardy and Ramanujan on the distribution of integers having a given number of prime factors},
J. Indian Math. Soc. {\bf 17} (1953), 63--141.

\bibitem{Sathe1954}
L. G. Sathe,
\textit{On a problem of Hardy and Ramanujan on the distribution of integers having a given number of prime factors},
J. Indian Math. Soc. {\bf 18} (1954), 27--81.

\bibitem{Selberg1954}
A. Selberg,
\textit{Note on the paper by L. G. Sathe},
J. Indian Math. Soc. {\bf 18} (1954), 83--87.

\bibitem{Tenenbaum1995}
G. Tenenbaum,
\textit{Introduction to analytic and probabilistic number theory},
Translated from the second French edition (1995) by C. B. Thomas,
Cambridge Studies in Advanced Mathematics {\bf 46},
Cambridge University Press, Cambridge, 1995. xvi+448 pp.

\bibitem{Titchmarsh1952}
E. C. Titchmarsh,
\textit{The theory of function}, Second edition,
Oxford University Press, Oxford, 1952.

\bibitem{Titchmarsh1986}
E. C. Titchmarsh,
\textit{The theory of the Riemann zeta-function},
Second edition, Revised by D. R. Heath-Brown,
Clarendon Press, Oxford, 1986. x+412 pp.

\bibitem{WuJWuQ2018}
J. Wu \& Q. Wu,
\textit{Beta-laws on divisors of square-full numbers and of integers representable as sums of two squares
in short intervals},
Preprint 2018.

\end{thebibliography}
\end{document}